\documentclass[11pt]{article}

\usepackage{hyperref}
\usepackage{graphicx} 

\makeatletter

\usepackage{amssymb,amsmath,amsthm,latexsym}
\usepackage[mathscr]{eucal} \usepackage{mathrsfs}
\usepackage{comment}
\usepackage{color}
\usepackage{hyperref}
\usepackage{graphicx}
\usepackage{array}
\usepackage{rotating}
 
\title{Generalizing a family of scattered quadrinomials in $\mathbb{F}_{q^{2t}}[X]$}
\author{Alessandro Giannoni, Giovanni Giuseppe Grimaldi,\\
Giovanni Longobardi, Marco Timpanella}
\date{}

\DeclareMathOperator{\Tr}{Tr}
\DeclareMathOperator{\N}{N}
\DeclareMathOperator{\PG}{{PG}}

\DeclareMathOperator{\GL}{{GL}}

\DeclareMathOperator{\GaL}{\Gamma L}
\DeclareMathOperator{\im}{im}

\begin{document}
\maketitle

\newtheorem{theorem}{Theorem}[section]
\newtheorem{lemma}[theorem]{Lemma}
\newtheorem{thm}[theorem]{Theorem}
\newtheorem{conj}[theorem]{Conjecture}
\newtheorem{remark}[theorem]{Remark}
\newtheorem{cor}[theorem]{Corollary}
\newtheorem{prop}[theorem]{Proposition}
\newtheorem{definition}[theorem]{Definition}
\newtheorem{result}[theorem]{Result}
\newtheorem{property}[theorem]{Property}
\newtheorem{conjecture}[theorem]{Conjecture}
\newtheorem{question}[theorem]{Question}

\makeatother
\newcommand{\Prf}{\noindent{\bf Proof}.\quad }
\renewcommand{\labelenumi}{(\alph{enumi})}
\newcommand{\la}{\langle}
\newcommand{\ra}{\rangle}
\newcommand{\G}{\mathrm{\Gamma L}}


\def\B{\mathbf B}
\def\C{\mathbf C}
\def\Z{\mathbf Z}
\def\Q{\mathbf Q}
\def\W{\mathbf W}
\def\a{\mathbf a}
\def\b{\mathbf b}
\def\c{\mathbf c}
\def\d{\mathbf d}
\def\e{\mathbf e}
\def\l{\mathbf l}
\def\v{\mathbf v}
\def\w{\mathbf w}
\def\x{\mathbf x}
\def\y{\mathbf y}
\def\z{\mathbf z}
\def\t{\mathbf t}
\def\cD{\mathcal D}
\def\cC{\mathcal C}
\def\cH{\mathcal H}
\def\cM{{\mathcal M}}
\def\cK{\mathcal K}
\def\cQ{\mathcal Q}
\def\cU{\mathcal U}
\def\cS{\mathcal S}
\def\cT{\mathcal T}
\def\cR{\mathcal R}
\def\cN{\mathcal N}
\def\cA{\mathcal A}
\def\cF{\mathcal F}
\def\cL{\mathcal L}
\def\cP{\mathcal P}
\def\cG{\mathcal G}
\def\cGD{\mathcal GD}
\def\qs{{q^s}}
\def\Fq{{\mathbb{F}_q}}
\def\F{{\mathbb{F}}}
\newcommand{\Fqn}{\F_{q^n}}
\newcommand{\Fqdd}{\F_{q^d}}
\newcommand{\Fqt}{{\F_{q^t}}}

\def\GF{{\rm GF}}
\def\rk{{\rm rk}}

\def\Pg{PG(5,q)}
\def\pg{PG(3,q^2)}
\def\ppg{PG(3,q)}
\def\HH{{\cal H}(2,q^2)}
\def\F{\mathbb F}
\def\Ft{\mathbb F_{q^t}}
\def\P{\mathbb P}
\def\V{\mathbb V}
\def\bS{\mathbb S}
\def\E{\mathbb E}
\def\N{\mathbb N}
\def\K{\mathbb K}
\def\D{\mathbb D}
\def\ps@headings{
 \def\@oddhead{\footnotesize\rm\hfill\runningheadodd\hfill\thepage}
 \def\@evenhead{\footnotesize\rm\thepage\hfill\runningheadeven\hfill}
 \def\@oddfoot{}
 \def\@evenfoot{\@oddfoot}
}
\def\cub{\mathscr C}
\def\cO{\mathcal O}
\def\cur{\mathscr L}
\def\Fqm{{\mathbb F}_{q^m}}
\def\Fq3{{\mathbb F}_{q^3}}
\def\fq{{\mathbb F}_{q}}
\def\Fm{{\mathbb F}_{q^m}}

\begin{abstract}
\noindent In recent years, several efforts have focused on identifying new families of scattered polynomials. Currently, only three families in $\F_{q^n}[X]$ are known to exist for infinitely many values of $n$ and $q$:  (i) pseudoregulus-type monomials, (ii) Lunardon–Polverino-type binomials, and (iii) a family of quadrinomials studied in a series of papers.
In this work, we provide sufficient conditions under which these quadrinomials, denoted by $\psi_{m,h,s}$, are scattered. Our results both include and generalize those obtained in previous studies. We also investigate the equivalences between the previously known families of scattered polynomials and those in this new class.

\end{abstract}

\noindent 

\vspace{0.4cm}

\noindent \thanks{{\textbf{Keywords:}   Linearized polynomials; MRD codes; Scattered polynomials;}

\section{Introduction}

Let $\mathbb{F}_{q^n}$ be the finite field with $q^n$ elements, where $n \geq 2$ and $q$ a prime power and consider $V=\mathbb{F}_{q^n}^r$ an $r$-dimensional vector space over $\mathbb{F}_{q^n}$. If $U$ is an $\F_q$-vector subspace of $V$ of dimension $u$, then  the set of points
$$
L=L_U=\{ \langle \textbf{u} \rangle_{\mathbb{F}_{q^n}} : \textbf{u} \in U \setminus \{ \textbf{0}\} \}
$$
in the projective space $\mathrm{PG}(r-1,q^n)=\mathrm{PG}(V, \mathbb{F}_{q^n})$ is called an \textit{$\mathbb{F}_q$-linear set} of \textit{rank $u$}. If $u=r$ and $\langle L_U \rangle=\PG(r-1,q^n)$, we will call $L_U$ a \textit{canonical subgeometry}. The size of $L$ can be at most $(q^u-1)/(q-1)$. If this number is achieved, the linear set is called \textit{scattered} and this is equivalent to say that
\begin{equation}\label{scattsub}
\dim_{\mathbb{F}_q} (U \cap \langle \textbf{v} \rangle_{\mathbb{F}_{q^n}}) \leq 1 \,\,\, \textrm{for any} \,\,\, \textbf{v} \in V.
\end{equation}
If condition \eqref{scattsub} holds, then $U$ is called \textit{$\F_q$-scattered subspace}. By \cite{BlockLaw}, any scattered linear set of $\mathrm{PG}(r-1,q^n)$ has rank at most $rn/2$, and when the rank attains the largest possible value both the linear set and the underlying $\F_q$-subspace are called \textit{maximum}. Two linear sets $L_U$ and $L_W$ of $\mathrm{PG}(r-1,q^n)$ are \textit{$\mathrm{P \Gamma L}$-equivalent} if there exists $\Phi \in \mathrm{P \Gamma L}(r,q^n)$ such that $L_U^{\Phi}=L_W$, and we will write $L_U \cong L_W$. It is clear that if $U$ and $W$ are two $\F_q$-subspaces in the same orbit under the action of the group $\mathrm{\Gamma L}(r,q^n)$, then the linear sets $L_U$ and $L_W$ are equivalent. On  the other hand, it is possible that $L_U=L_W$, but  $U$ and $W$ are not in the same $\GaL(r,q^n)$-orbit. For instance, if $1<s < n-1$ and $\gcd(s,n)=1$, then the $\F_q$-subspaces 
\[U=\{ (x,x^q):x \in \F_{q^n} \} \quad \text{ and } \quad  W=\{ (x,x^{q^s}): x \in \F_{q^n}\}\] are not $\mathrm{\Gamma L}$-equivalent, but they define the same linear set $L=\{ \langle (1,x^{q-1}) \rangle_{\F_{q^n}}: x \in \F_{q^n}^*\}$, see \cite{CsajbokZanella2016}.

In this article, we focus on the linear sets of the projective line $\PG(1,q^n)$. 
In the following, we introduce some notions in order to make the article as self-contained as possible and to show the connection of these objects with coding theory. 
For more details on the relation between MRD codes, scattered subspaces and linear sets, see \cite{survey, standard, PolZu, SheMRD} and the reference therein.\\ 

Throughout this paper, we will denote the \textit{trace} and the \textit{norm} functions of $ \F_{q^n}$ over $\F_{q^d}$, with $d \mid n$, by
$$\Tr_{q^n/q^d}(x) = \sum^{n/d-1}_{i=0} x^{q^{id}} \quad \textnormal{ and } \quad  \mathrm{N}_{q^n/q^d}(x) =x^\frac{q^n-1}{q^d-1},$$
respectively. For more details on their properties, see \cite[Chapter~2]{Lidl}.\\
Let $s$ be a non-negative integer such that $\gcd(s,n)=1$. A \textit{$q^s$-linearized polynomial over $\F_{q^n}$} is a polynomial in $\F_{q^n}[X]$ having the form
$$
f=\sum_{i=0}^{\ell} a_i X^{q^{si}}.
$$
 If $a_{\ell} \ne 0$, then the integer $\ell$ is called the \textit{$q^s$-degree} of $f$. The \textit{adjoint polynomial} of $f$ is defined as 
$$
\hat{f}= \sum_{i=0}^{\ell} a_i^{q^{s(n-i)}}X^{q^{s(n-i)}}.
$$ 
Let
 $$
 \cL_{n,q,s}=\left\{ \sum_{i=0}^{k} a_i X^{q^{si}} : a_i \in \F_{q^n}, \,\, k \in \mathbb{N} \right\}
 $$
be the set of $q^s$-linearized polynomials defined over $\F_{q^n}$. It is well known that 
 $$
 \tilde{\cL}_{n,q,s}:= \left\{ \sum_{i=0}^{n-1} a_i X^{q^{si}} : a_i \in \F_{q^n} \right\}
 $$
 equipped with the sum, the map composition  modulo $X^{q^{sn}}-X$ and the scalar multiplication by an element in $\F_{q}$ is an $\F_q$-algebra isomorphic to the set   $\mathrm{End}_{\F_q}(\F_{q^n})$ of all endomorphisms of $\F_{q^n}$ seen as an $\F_q$-vector space, see \cite[Chapter 3]{Lidl}. Given $f=\sum_{i=0}^{n-1}a_iX^{q^{si}} \in \tilde{\cL}_{n,q,s}$,
 we will denote by $f(x)$ the  map $x \in \F_{q^n} \longmapsto \sum_{i=0}^{n-1}a_i x^{q^{si}} \in \F_{q^n}$ of $\mathrm{End}_{\F_q}(\F_{q^n})$.

A $q^s$-linearized polynomial $f$ is \textit{scattered} if the polynomial $f+mX$ has at most $q$ roots in $\F_{q^n}$ for any $m \in \F_{q^n}$, and this is equivalent to saying that for any $y,z \in \F_{q^n}^*$ satisfying
\begin{equation}\label{scattpol}
\frac{f(y)}{y}=\frac{f(z)}{z}
\end{equation}
then $y,z$ must be $\F_q$-linearly dependent. Note that if $f \in \tilde{\cL}_{n,q,s}$ is scattered then also its adjoint $\hat{f}$ is scattered.

Up to a collineation of $\PG(1,q^n)$, we can always assume that a linear set $L_U \subset \PG(1,q^n)$ does not contain the point $\langle (0,1) \rangle_{\F_{q^n}}$ and so the subspace $U$ has the shape
$$
U=U_{f}=\{ (x,f(x)) : x \in \F_{q^n} \}
$$
for some $q^s$-linearized polynomial $f$. For this reason, the linear set $L_U$ will be also denoted by $L_{f}$. Note that $L_{f}$ is a scattered linear set if and only if $f$ is a scattered polynomial. 

Two $q^s$-linearized polynomials $f$ and $g$ are \textit{$\mathrm{\Gamma L}$-equivalent} (resp. \textit{$\mathrm{GL}$-equivalent}) if the subspaces $U_f$ and $U_g$ are in the same $\mathrm{\Gamma L}(2,q^n)$-orbit (resp. $\mathrm{GL}(2,q^n)$-orbit). In this case, we will write $f\sim_{\mathrm{\Gamma L}}g$ (resp. $f \sim_{\mathrm{GL}} g$).

\medskip
A \textit{rank-metric code} is a subset $\cC$ of the $\F_q$-vector space $\F_q^{m \times n}$ of the matrices with entries over $\F_q$ containing at least two elements and equipped with the \textit{rank distance}, defined as
$$
d(A,B)=\mathrm{rk}(A-B) \,\,\, \mathrm{for \,\,  any} \,\,\, A,B \in \F_{q}^{m \times n}.
$$
The \textit{minimum distance} $d(\cC)$ of $\cC$ is defined as
$$
d=d(\cC)=\min_{\underset{A \neq B}{A,B \in \cC}} d(A,B).
$$
The code $\cC$ will be called a rank-metric code, or for short \textit{RM code}, with parameters $(m,n,q;d)$. An RM code is \textit{linear} if it is an $\F_q$-subspace of $\F_{q}^{m \times n}$. It is well known that the size of a rank-metric code satisfies the \textit{Singleton-like bound} \cite[Theorem 5.4]{Singleton}, i.e.,
$$
|\cC| \leq q^{\max\{m,n\}(\min\{m,n\}-d+1)}.
$$
When this bound is attained, the code $\cC$ is called a \textit{maximum rank distance} code, or shortly \textit{MRD} code. Two linear rank-metric codes $\cC, \cC' \subseteq \F_{q}^{m \times n}$ are \textit{equivalent} if there exist $A \in \mathrm{GL}(m,q)$, $B \in \mathrm{GL}(n,q)$ and $\rho \in \mathrm{Aut}(\F_q)$ such that
$$
\cC'=A \cC^{\rho} B:= \{ AC^{\rho}B: C \in \cC \}.
$$
Since there is an isomorphism between $\tilde{\cL}_{n,q,s}$ and $\mathrm{End}_{\F_q}(\F_{q^n})$, a rank-metric code made up of $n \times n$ matrices can be seen as a subset of $\tilde{\cL}_{n,q,s}$. In this setting, the minimum distance of $\cC \subseteq \tilde{\cL}_{n,q,s}$ is defined as
$$
\min_{\underset{f,g \in \cC}{f \neq g}} \dim_{\F_q} \mathrm{im} ((f - g)(x)),
$$
and the equivalence between $\cC, \cC' \subseteq \tilde{\cL}_{n,q,s}$ can be stated as follows: there are $g,h \in \tilde{\cL}_{n,q,s}$ permutation polynomials and $\rho \in \mathrm{Aut}(\F_q)$ such that
$$
\cC'=g \circ \cC^{\rho} \circ h:=\{ g \circ f^{\rho} \circ h : f \in \cC \},
$$
where the automorphism $\rho$ acts only on the coefficients of the polynomial $f$. In general, it is rather difficult to determine whether two rank-metric codes are equivalent or not. For this reason, it is useful to introduce the  \textit{left} and \textit{right idealizers} of a rank-metric code $\cC \subseteq \tilde{\cL}_{n,q,s}$, that is
$$
I_{L}(\cC)=\{ g \in \tilde{\cL}_{n,q,s} \colon \, g \circ f \in \cC \quad \forall f \in \cC \}
$$
and 
$$I_{R}(\cC)=\{ h \in \tilde{\cL}_{n,q,s} \colon \, f \circ h \in \cC \quad \forall f \in \cC \},
$$
respectively. As shown in \cite[Proposition 4.1]{LTZ}, left and right idealizers are invariant under equivalence. Moreover, if $\cC$ is a linear MRD code, they are isomorphic to subfields of $\F_{q^n}$; see \cite{survey}.

Given $f \in \tilde{\cL}_{n,q,s}$, we can consider the subspace of $q^s$-linearized polynomials
$$
\cC_{f}=\{ aX+bf : a,b \in \F_{q^n} \}= \langle X, f \rangle_{\F_{q^n}}.
$$
It is easy to show that $\cC_{f}$ is a linear MRD code with minimum distance $n-1$ if and only if $f$ is a scattered polynomial. Moreover,  a code $\cC_f$ is an $\F_{q^n}$-subspace of $\tilde{\mathcal{L}}_{n,q,s}$ and the left idealizer  is 
$$I_L(\cC_{f})= \{ \alpha X \colon \alpha \in \F_{q^n} \}.$$
In \cite{JShe}, the author proved that given $f$ and $g$ two scattered linearized polynomials, the MRD codes $\cC_{f}$ and $\cC_{g}$ are equivalent if and only if $f \sim_{\GaL} g$.\\
Let  $f$ be a scattered polynomial and consider $G_f$ the stabilizer in $\mathrm{GL}(2,q^n)$ of the scattered subspace $U_f \subseteq \F_{q^n} \times \F_{q^n}$. Then, $G^{\circ}_f = G_f  \cup \{O\}$, where $O$ is the null matrix of order $2$, is a subfield of matrices isomorphic to $I_R(\cC_f)$ as field, \cite[Proposition 2.2]{survey}.\\

In literature, up to now, there are only three families of maximum scattered linear sets of the projective line $\PG(1,q^n)$ for infinitely many $n$ and $q$:
\begin{itemize}
    \item [$(i)$]  $f_{1,s}:=X^{q^s}$, known as \textit{pseudoregulus type}, where $1 \leq  s \leq  n-1$ and $\gcd(s,n) = 1$, see \cite{BlockLaw}. In this case we have:
    $$
    G_{f_{1,s}}^{\circ}=\left\{ \begin{pmatrix} \alpha & 0 \\ 0 & \alpha^{q^s}  \end{pmatrix} : \alpha \in \F_{q^n}\right\} \,\,\, \mathrm{and} \,\,\, I_R(\mathcal{C}_{f_{1,s}})=\{ \alpha X : \alpha \in \F_{q^n} \},
    $$
    see \cite{Cjabok}.
    \item [$(ii)$] $f_{2,s}:=X^{q^s} + \delta X^{q^{s(n-1)}}$, known as \textit{Lunardon-Polverino (LP) type}, where $n \geq 4$, $\mathrm{N}_{q^n/q}(\delta)\notin \{0,1\} $, $1 \leq s \leq n-1$  and $\gcd(s,n) = 1$, see \cite{LunPol,SheMRD}. Then,
     $$
    G_{f_{2,s}}^{\circ}=\left\{ \begin{pmatrix} \alpha & 0 \\ 0 & \alpha^{q^s}  \end{pmatrix} : \alpha \in \F_{q^{\gcd(2,n)}}\right\} \,\,\, \mathrm{and} \,\,\, I_R(\mathcal{C}_{f_{2,s}})=\{ \alpha X : \alpha \in \F_{q^{\gcd(2,n)}} \},
    $$
    see \cite{translation, Cjabok}.
    \item [$(iii)$]  
  \begin{equation}\label{quad}
      \psi_{m,h,s}:=  m \left (X^{q^s}-h^{1-q^{s(t+1)}}X^{q^{s(t+1)}} \right )+X^{q^{s(t-1)}}+h^{1-q^{s(2t-1)}}X^{q^{s(2t-1)}} \in \tilde{\mathcal{L}}_{n,q,s}, 
      \end{equation}
      where $q$ is an odd prime power, $t \geq 3$, $n=2t$, $\gcd(s,n)=1$, and $(m,h) \in \F_{q^t} \times \F_{q^{2t}}$ are such that 
      \begin{itemize}
    \item $m=1$ and $h \in \F_{q^t}$ with $h^2=-1$, see \cite{Bartoli_Zanella_Zullo,longobardi_zanella, neri_santonastaso_zullo,Zanella-Zullo}. Note that $h^2=-1$ implies that $h \in  \F_{q^{\gcd(t,2)}}$. So, if $t$ is odd, $q \equiv 1 \pmod 4$.

    \item $m=1$ and $h \in \F_{q^{2t}} \setminus \F_{q^t}$ with $\mathrm{N}_{q^{2t}/q^t}(h)=-1$, see \cite{Bartoli_Zanella_Zullo, longobardi_marino_trombetti_zhou,neri_santonastaso_zullo}.

    \item $h \in \F_q$ and $m \in \F_{q^t}$ such that it is neither a $(q+1)$-th nor $(q-1)$-th power of an element belonging to $\ker\, \mathrm{Tr}_{q^{2t}/q^t}$, see \cite{SmaZaZu}.
\end{itemize}
\end{itemize}

In this paper we will study some sufficient conditions for the polynomials $\psi_{m,h,s}$ in \eqref{quad} to be scattered. Our results will include and generalize those obtained in previous works, see \cite{Bartoli_Zanella_Zullo,longobardi_marino_trombetti_zhou,longobardi_zanella,  neri_santonastaso_zullo,SmaZaZu,Zanella-Zullo}. 
More precisely, define 
\begin{equation}\label{p+p-}
  \mathscr{P}^+_{s}= \left \{w^{q^s+1} \in \F_{q^t}  \colon w \in \ker \mathrm{Tr}_{q^{2t}/q^t} \right \} \quad \textnormal{ and } \quad \mathscr{P}^{-}_{s}=\left \{w^{q^s-1} \in \F_{q^t}  \colon w \in \ker \mathrm{Tr}_{q^{2t}/q^t} \right \}.
\end{equation}

Our main theorem is the following.

 \begin{thm}\label{main}
     Let $t \geq 3$, $q$ an odd prime power, $1 \leq s \leq 2t-1$ an integer such that $\gcd(s,2t)=1$, and $(m,h) \in \F_{q^t}^* \times \F_{q^{2t}}$. The polynomial $\psi_{m,h,s}$
 is scattered under one of the following assumptions:
 \begin{itemize}
     \item [$i)$] $t$ is even, or $t$ is odd and $q \equiv 1 \pmod 4$, and \begin{equation*}
            m \in \mathbb{F}_{q^t} \setminus (\mathscr{P}^+_{s} \cup \mathscr{P}^-_{s}) \quad \text{ with } \quad \mathrm{N}_{q^{2t}/q^t}(h)=\pm 1;
    \end{equation*}
     \item [$ii)$] $t$ is odd, $q \equiv 3 \pmod 4$ and either
     \begin{itemize}
         \item [$(a)$]  \begin{equation*}
         m \in \mathscr{P}^+_{s} \quad \text{ with } \quad  \mathrm{N}_{q^{2t}/q^t}(h)=-1, \,\, \text{ or } 
            \end{equation*}
         \item [$(b)$] \begin{equation*}
       m \in \mathbb{F}_{q^t} \setminus (\mathscr{P}^+_{s} \cup \mathscr{P}^-_{s})  \quad \text{ with } \quad \mathrm{N}_{q^{2t}/q^t}(h)=1 \text{ and } h^2 \neq -1.  
       \end{equation*}
     \end{itemize}
 \end{itemize}
 \end{thm}

The conditions in Theorem \ref{main} strictly include those introduced in the previous papers on $\psi_{m,h,s}$.\\

\noindent \textbf{Outline.} The paper is organized as follows. In Section \ref{ScattCond}, we will prove some preliminary results that will be used to prove Theorem \ref{main}. Section \ref{sec main} is devoted to the proof of Theorem \ref{main}, which will be divided in two separate cases, according to $\mathrm{N}_{q^{2t}/q^t}(h)=-1$ or $\mathrm{N}_{q^{2t}/q^t}(h)=1$. In Section \ref{sec:equiv} we will deal with the equivalence issue, and we show that the family of quadrinomials in Theorem \ref{main} is not equivalent to the three known families of scattered polynomials of $\F_{q^n}$ existing for infinitely many $n$ and $q$. Finally, in Sections \ref{sec:stab} and \ref{sec:necessarie} we compute the stabilizer of the subspace $U_{\psi_{m,h,s}}$ and provide partial results on the necessary conditions for the quadrinomial $\psi_{m,h,s}$ to be scattered.

\section{Preliminary Results}\label{ScattCond}
In this section we will collect some preliminary results that will be used to prove Theorem \ref{main}.
From now on, let $q$ be  an odd prime power, $t \geq 3$ and $1 \leq s \leq  2t-1$ with $\gcd(s,2t)=1$. We will consider the $q^s$-linearized polynomial of $\mathcal{\tilde{L}}_{2t,q,s}$
 \begin{equation}\label{quadrinomial}
 \psi_{m,h,s}=m(X^{q^s}-h^{1-q^{s(t+1)}}X^{q^{s(t+1)}})+X^{q^{s(t-1)}}+h^{1-q^{s(2t-1)}}X^{q^{s(2t-1)}}
 \end{equation}
where $(m,h) \in \F^*_{q^t} \times \F^*_{q^{2t}}$. The following result can be found in \cite[Lemma 2.1]{SmaZaZu}.

\begin{lemma}\label{power}
Let $\F_{q^{2t}}$ be the finite field of $q^{2t}$ elements, $t \geq 3$ and $q$ an odd prime power. Then,
\begin{itemize}
\item [$(i)$] the additive group of $\F_{q^{2t}}$ is the the direct sum of $\F_{q^t}$ and $\ker \mathrm{Tr}_{q^{2t}/q^t}$;
\item [$(ii)$] the product of two elements in $\ker \mathrm{Tr}_{q^{2t}/q^t}$ belongs to $\F_{q^t}$;
\item [$(iii)$] if $w\in \ker \mathrm{Tr}_{q^{2t}/q^t}$  and $\ell$ is a non-negative integer 
    \begin{equation*}
        w^\ell\in\begin{cases}
            \ker \mathrm{Tr}_{q^{2t}/q^t}  & \text{if}\,\,\ell\,\,\text{is}\,\,\text{odd},\\
            \mathbb{F}_{q^t} & otherwise.
        \end{cases}
    \end{equation*}
\end{itemize}
\end{lemma}

An important role in our proofs will be played by the sets $\mathscr{P}^+_{s}$ and $\mathscr{P}^-_{s}$ as in \eqref{p+p-}.
It is straightforward to see that both these sets are contained in $\mathbb{F}_{q^t}$. Moreover, the following result holds.

\begin{lemma}  \label{lem:P}
Let $q$ be an odd prime power, $t\geq 3$, and $\gcd(s,2t)=1$. 
Then
\begin{itemize}
\item [$(i)$]  $\mathscr{P}^{-}_{s}=\mathscr{P}^{-}_1$ and $\mathscr{P}^{+}_s=\mathscr{P}^{+}_1$;
\item [$(ii)$] $\mathscr{P}^{+}_s \cap \mathscr{P}^{-}_s= \{0\}$.
\end{itemize}
\end{lemma}
\begin{proof}
  $(i)$ Let $v:=(q^s-1)/(q-1)$ and $y \in \mathscr{P}^{-}_s$. Then, $y=w^{q^s-1}$ with $w \in \ker \mathrm{Tr}_{q^{2t}/q^t}$. Since $y=w^{q^s-1}=\left(w^v\right)^{q-1}$ and  $v$ is odd, by Lemma \ref{power}, $w^v$ belongs to $\ker \mathrm{Tr}_{q^{2t}/q^t}$ and hence $y \in \mathscr{P}^{-}_1$. Therefore, $\mathscr{P}^{-}_{s}\subseteq\mathscr{P}^{-}_1$. Conversely, let $w^{q-1}\in \mathscr{P}^-_1$ with $w\in \ker \mathrm{Tr}_{q^{2t}/q^t}$. Since $\gcd\left(v, \frac{q^{2t}-1}{q-1}\right)=1$, there exists $a \in \mathbb{Z}$ such that $av \equiv 1 \pmod{\frac{q^{2t}-1}{q-1}}$. Note that $a$ is odd and by Lemma \ref{power} we have $\overline{w}=w^a \in \ker \mathrm{Tr}_{q^{2t}/q^t}$. Then,
   \begin{eqnarray*}
   \overline{w}^{q^s-1}= w^{a(q^s-1)}=w^{av(q-1)}= w^{q-1},
   \end{eqnarray*} which implies $w^{q-1} \in \mathscr{P}^{-}_s$.\\ Putting $v:=(q^s+1)/(q+1)$ and noting that $\gcd\left(v, \frac{q^{2t}-1}{q+1}\right)=1$,  a similar argument can be applied in order to get $\mathscr{P}^+_{1}=\mathscr{P}^+_{s}$.  \\
$(ii)$ By $(i)$, we may prove the statement for $s=1$. Let $x,y\in \ker \mathrm{Tr}_{q^{2t}/q^t}$ such that $x^{q+1}=y^{q-1}$. As $q$ is odd, we have 
$$\left(x^{\frac{q+1}{2}}\right)^2=\left(y^{\frac{q-1}{2}}\right)^2,$$
and hence $x^{\frac{q+1}{2}}= \pm y^{\frac{q-1}{2}}$.\\
If $x^{\frac{q+1}{2}}=y^{\frac{q-1}{2}}$, by raising to the $q^t$-th power both sides of this equation, we get
        $$(-x)^{\frac{q+1}{2}}=(-y)^{\frac{q-1}{2}}.$$
Since only one between $(q+1)/2$ and $(q-1)/2$ is odd, we obtain $x^{\frac{q+1}{2}}=-y^{\frac{q-1}{2}}=-x^{\frac{q+1}{2}}$, which implies $x=y=0$.
The same holds if $x^{\frac{q+1}{2}}=-y^{\frac{q-1}{2}}$.
\end{proof}

The next results can be deduced from \cite[Proposition 3.2-Proposition 3.6]{longobardi_marino_trombetti_zhou}. 
To make this article as self-contained as possible, we will recall them below.

\begin{prop} \label{p:hcondition}
Let $t\geq 3$ and $h \in \F_{q^{2t}}$.
\begin{itemize}
    \item [$(i)$] If $\mathrm{N}_{q^{2t}/q^t}(h)=-1$, then $h^{q^{2s}+1} \neq 1$.
    \item [$(ii)$] If $\mathrm{N}_{q^{2t}/q^t}(h)=1$ with $h^2 \neq -1$, then $h^{q^{2s}+1} \neq -1$.
\end{itemize}
Moreover, in both cases $h^{q^{s(t-2)}} \neq -h$. 
\end{prop}

\begin{proof}

$(i)$ See  \cite[Proposition 3.2]{longobardi_marino_trombetti_zhou}.\\
Regarding $(ii)$, suppose that $\mathrm{N}_{q^{2t}/q^t}(h)=1$. We will split the proof in two cases.\\
\textit{Case 1:} $q \equiv 1 \pmod 4$. Let $k= \lambda h$ with $\lambda^2=-1$. By $(i)$, then $k^{q^{2s}+1} \neq 1$. This implies that $h^{q^{2s}+1} \neq -1$.\\
\textit{Case 2:} $q \equiv 3 \pmod 4$. If $t$ is odd, since $\gcd(q^{2s}+1,q^{2t}-1)=2$,  $h^2=-1$ if and only if $h^{q^{2s}+1}=-1$. Now, assume $t$ even and $h^{q^{2s}+1}=-1$. Then, $h \in \F_{q^4}$. If $t \equiv 0 \pmod 4$, then $h \in \F_{q^t}$ and $\mathrm{N}_{q^{2t}/q^t}(h)=h^2=1=-h^{q^{2s}+1}$, a contradiction. If $t \equiv 2 \pmod 4$, then $\mathrm{N}_{q^4/q^2}(h)=1=-h^{q^{2s}+1}$, a contradiction again. \\
Finally, $h^{q^{s(t-2)}} \neq -h$ as shown in \cite[Proposition 3.2]{longobardi_marino_trombetti_zhou}.
\end{proof}

Note that if $q \equiv 1 \pmod 4$ or  $t$ even and $q \equiv 3 \pmod 4$, an element $h \in \F_{q^{2t}}$ such that $\mathrm{N}_{q^{2t}/q^t}(h)=1$ and $h^2 = -1$ does not exist.

Throughout the paper it will be sometimes useful to write the polynomial $\psi_{m,h,s}$ as follows:
 \begin{equation*}
\psi_{m,h,s} = L_m + M,
\end{equation*}
where 
$$L_m = m(X^{q^s}-h^{1-q^{s(t+1)}}X^{q^{s(t+1)}})\quad \textnormal{ and } \quad M = X^{q^{s(t-1)}}+h^{1-q^{s(2t-1)}}X^{q^{s(2t-1)}}.$$
Putting $L:=L_1$, we have that $\ker L_m= \ker L$ and $\im L_m= \im L$ for any $m \in \F^*_{q^t}$. Moreover, if $\mathrm{N}_{q^{2t}/q^t}(h)= \pm 1$, 
\begin{equation}\label{im-1}
    \im L = \left \{z  \in \F_{q^{2t}} : z^{q^{st}} +h^{q^{st}-q^s}z = 0 \right \} \quad \textnormal{ and } \quad \im M=\left \{z \in \F_{q^{2t}} : z^{q^{st}}-h^{q^{st}-q^{s(t-1)}}z = 0 \right \}. 
\end{equation}

Although the next results are proven in \cite{longobardi_marino_trombetti_zhou} for $\mathrm{N}_{q^{2t}/q^t}(h)=-1$, these hold for $\mathrm{N}_{q^{2t}/q^t}(h)=1$ and $h^2\neq-1$, as well. Since the techniques are similar, we omit the proof. 

\begin{prop}\label{directsum}
The finite field $\F_{q^{2t}}$, $t \geq 3$, seen as $\Fqt$-vector space, is both the direct sum of $\ker L_m$ and  $\ker M$, and of $\im L_m$ and  $\im M$.
\end{prop}

Define the following $\Fqt$-linear maps of $\F_{q^{2t}}$
\begin{equation}\label{R-T-maps}
    R=X^{q^{st}}+h^{q^{s(t-1)}-q^s}X\quad\textnormal{and}\quad T=X^{q^{st}}+h^{q^s-q^{s(t-1)}}X.
\end{equation}
Clearly, $\dim_{\Fqt} \ker R = \dim_{\Fqt}\ker T= 1$ and it is straightfoward to see that  $\ker T = h^{q^{s(t-1)}-q^s} \ker R$.

\begin{lemma}\label{bases+split}
Let $\rho, \tau \in \F^*_{q^{2t}}$ and $t \geq 3$ such that $\rho \in \ker  R$ and $\tau \in \ker T$.
Then
\begin{enumerate}
    \item[$(i)$] $\{1,\rho\}$ and $\{1,\tau\}$ are $\F_{q^t}$-bases of $\F_{q^n}$.
    \item[$(ii)$] If $\tau=h^{q^{s(t-1)}-q^s} \rho$ and an element $\gamma \in \F_{q^{2t}}$ has components ($\lambda$,$\mu$) in the $\F_{q^t}$-basis $\{1,\rho\}$, then the components of  $\gamma$ in $\{1,\tau\}$ are
    \begin{equation}\label{components}
        \left(\lambda+ \mu\rho\left(1-h^{q^{s(t-1)}-q^s}\right), \mu\right).
    \end{equation}
\end{enumerate}
\end{lemma}

We conclude this section with the following results.

\begin{prop}\label{product-a}
For any nonzero vectors $u \in \ker L_m$, $v \in \ker M$ and any $a \in \F_{q^{2t}}$, the following statements are equivalent:
\begin{enumerate}
    \item[$(i)$]  $a \in \ker R$;
    \item[$(ii)$] $av \in \ker L_m$;
    \item[$(iii)$] $a M(u) \in \im\,L_m$.
\end{enumerate}
\end{prop}

\begin{prop}\label{product-b} 
For any nonzero vectors $u \in \ker L_m$, $v \in \ker M$ and any $b \in \F_{q^{2t}}$, the following statements are equivalent:
\begin{enumerate}
    \item[$(i)$] $b \in \ker T$;
    \item[$(ii)$] $b\,u \in \ker M$;
    \item[$(iii)$] $ b L(v) \in \im\, M$.
\end{enumerate}
\end{prop}

\begin{prop} \label{prop-W}
Let $x_1 \in \ker L_m$, $ 0 \neq x_2 \in \ker M$ and $h \in \F_{q^{2t}} \setminus \F_{q^t}$ such that $\mathrm{N}_{q^{2t}/q^t}(h)=-1$ or $\mathrm{N}_{q^{2t}/q^t}(h)=1$ and $h^2 \neq -1$. Then, the element 
$$\frac{M(x_1)}{x_2 \left (h^{q^s}+h^{q^{s(t-1)}} \right )}$$
belongs to $\ker \mathrm{Tr}_{q^{2t}/q^t}$.
\begin{proof}
Let $T_s :x \in \F_{q^{2t}} \longmapsto x+x^{q^{st}} \in \F_{q^t}$ and note that $\ker T_s = \ker \mathrm{Tr}_{q^{2t}/q^t}$. Then, in order to prove the statement, we have to show that 
$$T_s  \left ( \frac{M(x_1)}{x_2 \left (h^{q^s}+h^{q^{s(t-1)}} \right )} \right ) $$
 is equal to zero.\\
Suppose that $\mathrm{N}_{q^{2t}/q^t}(h)=-1$ . Since  $x_2\in\ker M$ and by \eqref{im-1}, we have
\begin{equation*}
\begin{split}
    &\left(\frac{M(x_1)}{x_2(h^{q^s}+h^{q^{s(t-1)}})} \right )^{q^{st}}+\frac{M(x_1)}{x_2(h^{q^s}+h^{q^{s(t-1)}})}\\
    &=-\frac{h^{q^{st}-q^{s(t-1)}}M(x_1)}{x_2h^{q^s-1}(h^{q^{s(t+1)}}+h^{q^{s(2t-1)}})} +\frac{M(x_1)}{x_2(h^{q^s}+h^{q^{s(t-1)}})}\\
    &=\frac{M(x_1)}{x_2h^{q^s+q^{s(t-1)}}(h^{q^{s(t+1)}}+h^{q^{s(2t-1)}})} +\frac{M(x_1)}{x_2(h^{q^s}+h^{q^{s(t-1)}})}=0\\
    \end{split}
\end{equation*}
A similar argument applies for $\mathrm{N}_{q^{2t}/q^t}(h)=1$ and $h^2 \neq -1$.
\end{proof}
\end{prop}

\section{Proof of Theorem \ref{main}}\label{sec main}

The aim of this section is to prove Theorem \ref{main}. To this end, we consider two separate cases, depending on whether $\mathrm{N}_{q^{2t}/q^t}(h) = -1$ or $\mathrm{N}_{q^{2t}/q^t}(h) = 1$.
First, we address the case where $h$ is in $\F_{q^t}$.

\begin{remark}\label{hfqt}
 \textnormal{Note that if $h \in \F_{q^t}$ with $\mathrm{N}_{q^{2t}/q^t}(h)=\pm 1$ and  $m \in \F_{q^t} \setminus (\mathscr{P}_s^+ \cup \mathscr{P}_s^-)$,  the quadrinomial $\psi_{m,h,s}$ is scattered, indeed:
    \begin{itemize}
        \item [$(i)$] If $h^2=-1$, then $h \in \F_{q^{\gcd(t,2)}}$. For $q \equiv 1 \pmod{4}$, $h \in \F_q$ and 
        $$
        \psi_{m,h,s}=m\Big( X^{q^s}-X^{q^{s(t+1)}}\Big)+X^{q^{s(t-1)}}+X^{q^{s(2t-1)}}.
        $$
         is scattered, see \cite{SmaZaZu}. For $q \equiv 3 \pmod{4}$, then $t$ is even, $h \in \F_{q^2} \setminus \F_q$ and
        $$
        \psi_{m,h,s}=m\Big( X^{q^s}+X^{q^{s(t+1)}}\Big)+X^{q^{s(t-1)}}-X^{q^{s(2t-1)}}.
        $$
        This is scattered if  and only if its adjoint polynomial  $\hat{\psi}_{m,h,s}$ is scattered. Since 
        $$
       \hat{\psi}_{m,h,s}= \mu \psi_{-\mu,h,s}=-\mu \Big( X^{q^s}-X^{q^{s(t+1)}}  \Big)+X^{q^{s(t-1)}}+X^{q^{s(2t-1)}},
        $$ where $\mu=1/m^{q^{s(t-1)}}$, this is scattered because $\mu \in \F_{q^t} \setminus (\mathscr{P}_s^+ \cup \mathscr{P}_s^-)$.
        \item [$(ii)$] if $h^2=1$, then $h = \pm{1}$. In this case, 
        $$
\psi_{m,h,s}=m\Big( X^{q^s}-X^{q^{s(t+1)}}\Big)+X^{q^{s(t-1)}}+X^{q^{s(2t-1)}}
        $$
    and it is scattered again.
    \end{itemize}}
\end{remark}

\subsection{Case \texorpdfstring{$\mathrm{N}_{q^{2t}/q^t}(h)=-1$}{N(h) = -1}}

In this section, we prove Theorem \ref{main} under the assumption $h \in \F_{q^{2t}} \setminus \F_{q^t}$ and $\mathrm{N}_{q^{2t}/q^t}(h)=-1$.

\begin{theorem}\label{t:main}
Let $t\geq 3$, $q$ be an odd prime power. For each $h \in \F_{q^{2t}} \setminus \F_{q^t}$ with $\mathrm{N}_{q^{2t}/q^t}(h)=-1$ and  $$m  \in 
\begin{cases}
\mathbb{F}_{q^t} \setminus \bigl(\mathscr{P}_s^+ \cup \mathscr{P}_s^- \bigr),
& \text{if $t$ is even, or $t$ is odd and $q \equiv 1 \pmod{4}$}, \\[2mm]
\mathscr{P}_s^+,
& \text{if $t$ is odd and $q \equiv 3 \pmod{4}$,}
\end{cases}
$$
the $q^s$-linearized polynomial
 \begin{equation*}
 \psi_{m,h,s}=m(X^{q^s}-h^{1-q^{s(t+1)}}X^{q^{s(t+1)}})+X^{q^{s(t-1)}}+h^{1-q^{s(2t-1)}}X^{q^{s(2t-1)}}
 \end{equation*}
 is scattered.
\end{theorem}

\begin{proof}
Let $\psi(x):=\psi_{m,h,s}(x)$ and note that the property of being a scattered polynomial can be rephrased in the following way:  for each $x\in\F_{q^n}^*$ and for each $\gamma\in\F_{q^n}$ such that 
\begin{equation}\label{e:psirx}
\psi(\gamma x)=\gamma \psi(x)
\end{equation}
then $\gamma \in\F_q$ holds.
Recall that 
\begin{equation*}
\psi(x)=L_m(x)+M(x),
\end{equation*}
and by Proposition \ref{directsum}
any $x \in\F_{q^n}$ can be uniquely written as $x=x_1+x_2$, where $x_1 \in \ker L_m$ and $x_2 \in \ker M$. Similarly, by Lemma \ref{bases+split} , if $\gamma \in \F_{q^n}$ there are exactly two elements $\lambda_1,\mu_1 \in \F_{q^t}$ and two elements $\lambda_2, \mu_2 \in \F_{q^t}$  such that $$\lambda_1+\mu_1\rho=\gamma= \lambda_2+\mu_2\tau$$ where
$\rho \in \ker R$ and $\tau= h^{q^{s(t-1)}-q^s}\rho\in\ker T$. If $a=\mu_1\rho$ and $b=\mu_2\tau$, then $a\in \ker R$, $b\in \ker T$, and Condition \eqref{e:psirx} may be re-written as follows
\begin{equation}\label{notsplit}
\begin{split}
L_m((\lambda_1+a)(x_1+x_2))+&M((\lambda_2+b)(x_1+x_2))=\\
& = (\lambda_2+b)L_m(x_1+x_2)+(\lambda_1+a)M(x_1+x_2).
\end{split}
\end{equation}
Since $x_1 \in \ker L_m$ and $x_2 \in \ker M$, it follows from  Propositions \ref{product-a}$(ii)$ and \ref{product-b}$(ii)$ that
\begin{equation*}
\begin{split}
 L_m(\lambda_1 x_2)+L_m(a x_1)&+M(\lambda_2x_1) +M(b x_2)=\lambda_2L_m(x_2)+b L_m(x_2)+\lambda_1M(x_1)+a M(x_1).
\end{split}
\end{equation*}
and hence
\begin{equation}\label{splitted}
\begin{split}
 \lambda_1 ^{q^s}L_m(x_2)+L_m(a x_1)&-\lambda_2L_m(x_2)-a M(x_1)= b L_m(x_2)+\lambda_1M(x_1)-\lambda_2^{q^{s(t-1)}}M(x_1)-M(b x_2).
\end{split}
\end{equation}
By Point $(iii)$ of Propositions 
\ref{product-a} and \ref{product-b}, the expressions on the left and right hand sides of \eqref{splitted} belong to $\im L_m$ and $\im  M$, respectively. So, by Proposition \ref{directsum}, we have 
\begin{equation*}
\begin{cases}
L_m(ax_1)-aM(x_1)=(\lambda_2-\lambda^{q^s}_1)L_m(x_2)\\
bL_m(x_2)-M(bx_2)=(\lambda_2^{q^{s(t-1)}}-\lambda_1)M(x_1).
\end{cases}
\end{equation*}
Raising to the $q^s$-th power the second equation, we get
\begin{equation}\label{system1}
\begin{cases}
L_m(ax_1)-aM(x_1)=(\lambda_2-\lambda^{q^s}_1)L_m(x_2)\\
b^{q^s}L_m(x_2)^{q^s}-M(bx_2)^{q^s}=(\lambda_2-\lambda_1^{q^s})M(x_1)^{q^s}.
\end{cases}
\end{equation}
Since $a=\mu_1 \rho$, $b=\mu_2\tau$ and $\tau=h^{q^{s(t-1)}-q^s}\rho$, by Lemma \ref{bases+split}, it follows
\begin{equation*}
\lambda_1+\mu_1\rho=\lambda_1+\mu_1\rho(1-h^{q^{s(t-1)}-q^s})+\mu_1\tau=\lambda_2+\mu_2\tau.\end{equation*}
Since $\{1,\tau\}$ is an $\F_{q^t}$-basis of $\F_{q^{2t}}$ and $\rho(1-h^{q^{s(t-1)}-q^{s}})$ belongs to $\F_{q^t}$, we get $\mu_1=\mu_2$ and $b=h^{q^{s(t-1)}-q^s}a$. 
First, we deal with the case $a=0$. In this case, $\mu_1=0$ and hence $\gamma=\lambda_1=\lambda_2\in\F_{q^t}$. Now, if $\lambda_1\neq \lambda_1^{q^s}$ then \eqref{system1} yields $L_m(x_2)=M(x_1)=0$. By Proposition \ref{directsum}, $x=x_1=x_2=0$, a contradiction. Then $\lambda_1=\lambda_2=\lambda_1^{q^s}$, which gives $\lambda_1\in\F_q$, i.e. $\gamma \in\F_q$ holds. 

Therefore, in the remainder of the proof  we assume $a\neq 0$. This implies $\gamma\in\F_{q^{2t}}\setminus\F_{q^t}$.
In order to deal with the case $a\neq 0$, we shall split our discussion into three cases.
\medskip

\noindent\textbf{Case 1:} $x_1=0$. In this case System \eqref{system1} reads
    \begin{equation*}
	\begin{cases}\label{x_1=0}
	(\lambda_2-\lambda^{q^s}_1)L_m(x_2)=0\\
	b^{q^s}L_m(x_2)^{q^s}-M(bx_2)^{q^s}=0.
	\end{cases}
	    \end{equation*}
{
From the second equation we have 
\begin{equation}\label{case1.1}
	   b^{q^s} =\frac{M(bx_2)^{q^s}}{L_m(x_2)^{q^s}}=\frac{b^{q^{st}}x_2^{q^{st}}+h^{q^s-1}bx_2}{m^{q^s}\left(x_2^{q^{s}}-h^{1-q^{s(t+1)}}x_2^{q^{s(t+1)}}\right)^{q^s}}.
	\end{equation}
Now, by the assumptions $b\in \ker T$, $x_2\in \ker M$ and $h^{q^t+1}=-1$, i.e.
$$b^{q^{st}}=-h^{q^s-q^{s(t-1)}}b, \quad x_2^{q^{st}}=-h^{q^{s}-1}x_2, \quad h^{q^{st}}=-1/h.
$$
Then, \eqref{case1.1} reads as
\begin{equation}\label{case1.2}
	   b^{q^s} =\frac{h^{q^s-1}bx_2(1+h^{q^s-q^{s(t-1)}})}{m^{q^s}\left(x_2^{q^{s}}-h^{q^{2s}+1}x_2^{q^{s}}\right)^{q^s}}=\frac{h^{q^s-1}bx_2(1+h^{q^s-q^{s(t-1)}})}{m^{q^s}x_2^{q^{2s}}\left(1+h^{q^{s}-q^{s(t-1)}}\right)^{q^{2s}}},
	\end{equation}
    and so
\begin{equation}\label{case1}
    b^{q^s-1}=\frac{1}{m^{q^s}}\left (\frac{h}{ \left (x_2(1+h^{q^s-q^{s(t-1)}} )\right )^{q^s+1}} \right ) ^{q^s-1}.
\end{equation}
}
Therefore, there exists $z\in\mathbb{F}_{q^{2t}}^*$ such that $z^{q^s-1}=m$. Then $z^{q^{st}-1}=m^{\frac{q^{st}-1}{q^s-1}}$ and hence $z^{q^{st}}=\mu z$ for some $\mu\in\mathbb{F}_{q}^*$, with  $\mu \ne -1$ (otherwise $m \in \mathscr{P}_s^-$).
From (\ref{case1}) we obtain
that there exists $\lambda \in \mathbb{F}_{q}^*$ such that 
	\begin{equation*}
	    b=\lambda \cdot \frac{h}{z^{q^s}\left(x_2(1+h^{q^s-q^{s(t-1)}})\right)^{q^s+1}}.
	\end{equation*}
	Since $b \in \ker T$, then $b^{q^{st}}+h^{q^s-q^{s(t-1)}}b=0$. By Formula above and since $z^{q^{st}}=\mu z$, we have that
	\begin{equation*}
	\frac{h^{q^{st}}}{\mu z^{q^s}\left(x_2(1+h^{q^s-q^{s(t-1)}})\right)^{q^{st}(q^s+1)}} +   \frac{h^{1+q^s-q^{s(t-1)}}}{z^{q^s}\left(x_2(1+h^{q^s-q^{s(t-1)}})\right)^{q^s+1}}=0.
	\end{equation*}
Since $x_2 \in \ker M$ and $h \ne 0$,
	\begin{equation*}
	\Biggl( \frac{1+h^{q^s-q^{s(t-1)}}}{
	h^{q^s-1}(1+h^{q^{s(t-1)}-q^s})} \Biggr )^{q^s+1}=-\mu h^{1+q^s-q^{s(t-1)}-q^{st}}.
	\end{equation*}
	This is equivalent to
	\begin{equation*}
	\Biggl( \frac{h^{q^{s(t-1)}-1}(1+h^{q^s-q^{s(t-1)}})}{
	h^{q^s-1}(1+h^{q^{s(t-1)}-q^s})} \Biggr )^{q^s+1}=-\mu.
	\end{equation*} 
 On the other hand, we have $$\frac{h^{q^{s(t-1)}-1}(1+h^{q^s-q^{s(t-1)}})}{
	h^{q^s-1}(1+h^{q^{s(t-1)}-q^s})}  =1.$$
Then $\mu = -1$, getting a contradiction.
\medskip

\noindent\textbf{Case 2:} $x_2=0$. 
In this case \eqref{system1} reads
	   \begin{equation*}
	\begin{cases}
	L_m(ax_1)-aM(x_1)=0\\
	(\lambda_2-\lambda_1^{q^s})M(x_1)^{q^s}=0.
	\end{cases}
	    \end{equation*}
        From the first equation we have 
\begin{equation}\label{case2.1}
	   m(a^\qs x_1^\qs-h^{1-q^{s(t+1)}}a^{q^{s(t+1)}}x_1^{q^{s(t+1)}})-a(x_1^{q^{s(t-1)}}+h^{1-q^{s(2t-1)}}x_1^{q^{s(2t-1)}})=0.
	\end{equation}
        Now, by the assumptions $a\in \ker R$, $\mathrm{N}_{q^{2t}/q^t}(h)=-1$ and $x_1\in \ker L_m$, we have
\begin{equation*}
	    m(a^\qs x_1^\qs+h^{1+q^{s}}(h^{-1-q^{2s}}a^\qs)(-h^{-1-\qs}x_1^\qs))-a(x_1^{q^{s(t-1)}}-h^{1+q^{s(t-1)}}(-h^{-q^{s(t-2)}-q^{s(t-1)}}x_1^{q^{s(t-1)}})=0,
	\end{equation*}
and so 
\begin{equation}\label{case2}
	    a^{\qs-1}=x_1^{q^{s(t-1)}-\qs}\frac{(1+h^{1-q^{s(t-2)}})}{m(1-h^{-1-q^{2s}})}=\frac{\left(x^{\qs}_1(1-h^{-q^{2s}-1})\right)^{q^{s(t-2)}-1}}{m},
	\end{equation}

 If $m$ is not a $(q^s-1)$-power, we have a contradiction. So let $z\in\mathbb{F}_{q^{2t}}^*$ such that $z^{q^s-1}=m$. By the hypothesis on $m$,  $z^{q^{st}}=\mu z$ with $\mu\in\mathbb{F}_{q}^*$ for some $\mu \in \F_q^*$ with $ \mu\neq -1$.
By  \eqref{case2}, we obtain
that there exists $\lambda \in \mathbb{F}_{q}^*$ such that 
	\begin{equation*}
	a=\lambda \frac{\left(x^\qs_1(1-h^{-q^{2s}-1})\right)^\nu}{z},
	\end{equation*}
	where $\nu=(q^{s(t-2)}-1)/(q^s-1)$.
	
	Since $a \in \ker R$ and $h^{-q^{s(t+2)}}=-h^{q^{2s}}$, then
	\begin{equation*}
	    \frac{\left(x_1^{q^{s(t+1)}}(1+h^{q^{2s}-q^{st}})\right)^\nu}{\mu z}+h^{q^{s(t-1)}-q^s}\frac{\left(x_1^\qs(1-h^{-q^{2s}-1})\right)^\nu}{z}=0.
	\end{equation*}
	Moreover, since $x_1 \in \ker L_m$, $x_1^{q^{s(t+1)}}=h^{q^{s(t+1)}-1}x_1^{q^s}$, and so
	\begin{equation*}
	     \Biggl (\ \frac{h^{q^{s(t+1)}-1} (1+h^{q^{2s}-q^{st}})}{1+h^{q^{st}-q^{2s}}}  \Biggr)^\nu=-\mu h^{q^{s(t-1)}-q^s},
	\end{equation*}
    or equivalently
	  \[ \Biggl (\ \frac{h^{q^{st}-\qs} (1+h^{q^{2s}-q^{st}})}{h^{q^{2s}-\qs}(1+h^{q^{st}-q^{2s}})}  \Biggr)^\nu=-\mu. \]
   On the other hand, 
   \[ \Biggl (\ \frac{h^{q^{st}-\qs} (1+h^{q^{2s}-q^{st}})}{h^{q^{2s}-\qs}(1+h^{q^{st}-q^{2s}})}  \Biggr)^\nu=1. \]
	Since $\mu\neq -1$, we have a contradiction. 

\noindent\textbf{Case 3:} $x_1,x_2 \neq 0$. Recall that $a \in \ker R$, $b=h^{q^{s(t-1)}-\qs}a$, $\lambda_2 = \lambda_1+(1-h^{q^{s(t-1)}-\qs})a$, $x_1 \in \ker L_m$ and $x_2 \in \ker M$. Then, by \eqref{system1}, $a$ turns out to be a nonzero solution of the following linear system in the unknowns $a,a^{q^s}$:

\begin{equation}\label{complete-linear-system}
    \begin{cases}
mx_1^\qs(1+h^{q^{st}-q^{2s}})a^\qs-\left(M(x_1)+(1-h^{q^{s(t-1)}-\qs})L_m(x_2)\right)a=(\lambda_1-\lambda^\qs_1)L_m(x_2)\\
h^{q^{st}-q^{2s}}L_m(x_2)^\qs a^\qs+\left(x_2^{q^{st}}(1+h^{q^{s(t-1)}-\qs})-(1-h^{q^{s(t-1)}-\qs})M(x_1)^\qs\right)a=(\lambda_1-\lambda_1^\qs)M(x_1)^\qs.
\end{cases}
\end{equation}

\noindent\textit{- Case 3.1}. First of all, suppose that $\lambda_1 \in \mathbb{F}_q$, then System \eqref{complete-linear-system} becomes

\begin{equation}\label{eq:case3.1_star}
    \begin{cases}
mx_1^\qs(1+h^{q^{st}-q^{2s}})a^\qs-\left(M(x_1)+(1-h^{q^{s(t-1)}-\qs})L_m(x_2)\right)a=0\\
h^{q^{st}-q^{2s}}L_m(x_2)^\qs a^\qs+\left(x_2^{q^{st}}(1+h^{q^{s(t-1)}-\qs})-(1-h^{q^{s(t-1)}-\qs})M(x_1)^\qs\right)a=0.
\end{cases}
\end{equation}
Since $a$ is a nonzero solution then
\begin{eqnarray}\label{determinant-incomplete}
 mx_1^\qs(1+h^{q^{st}-q^{2s}})\left(x_2^{q^{st}}(1+h^{q^{s(t-1)}-\qs})-(1-h^{q^{s(t-1)}-\qs})M(x_1)^\qs\right)=\nonumber\\=
 -h^{q^{st}-q^{2s}}L_m(x_2)^\qs\left(M(x_1)+(1-h^{q^{s(t-1)}-\qs})L_m(x_2)\right).
\end{eqnarray}

Since $L_m(x_2) \neq 0 \neq M(x_1)$, from \eqref{eq:case3.1_star} we get
$$M(x_1)^\qs\left(mx_1^\qs(1+h^{q^{st}-q^{2s}})a^\qs-M(x_1)a\right)=L_m(x_2)\left(h^{q^{st}-q^{2s}}L_m(x_2)^\qs a^\qs+x_2^{q^{st}}(1+h^{q^{s(t-1)}-\qs})a\right),$$
and equivalently
\begin{align}
\nonumber \left(mx_1^\qs M(x_1)^\qs(1+h^{q^{st}-q^{2s}}) - h^{q^{st}-q^{2s}}L_m(x_2)L_m(x_2)^\qs\right)a^\qs=\\
=\left(M(x_1)M(x_1)^\qs+x_2^{q^{st}}(1+h^{q^{s(t-1)}-\qs})L_m(x_2)\right)a. \label{Fq-case}
\end{align}

Next we want to show that the coefficient of $a^\qs$ in \eqref{Fq-case} cannot be $0$. By way of contradiction, suppose that
\begin{equation}\label{D2}
	mx_1^\qs(1+h^{q^{st}-q^{2s}})M(x_1)^\qs=h^{q^{st}-q^{2s}}L_m(x_2)L_m(x_2)^{q^s}.
\end{equation}

We have that also the coefficient of $a$ is $0$, from it we obtain
\begin{equation}\label{eq:m(q+1)}
    \begin{split}
         m&=-\frac{M(x_1)^{q^s+1}}{x_2^{q^{st}+q^s}(1+h^{1-q^{s(t+2)}})(1+h^{q^{s(t-1)}-q^s})}\\
         &=\frac{M(x_1)^{q^s+1}}{x_2^{q^s+1}h^{q^s-1}(1+h^{q^{s(t-1)}-q^s})(1+h^{1-q^{s(t+2)}})}\\
         &=\frac{M(x_1)^{q^s+1}}{x_2^{q^s+1}h^{q^s-1}(1+h^{q^{s(t-1)}-q^s})(1+h^{q^{2s}-q^{st}})}=\\
         &=\frac{M(x_1)^{q^s+1}}{x_2^{q^s+1}h^{q^s-1}(1+h^{q^{s(t-1)}-q^s})(1+h^{q^{s}-q^{s(t-1)}})^{q^s}}\\
         &=-\frac{M(x_1)^{q^s+1}}{x_2^{q^s+1}(h^{q^{s(t-1)}}+h^{q^{s}})^{q^s+1}}.    
    \end{split}
\end{equation}

If  $t$ is even, $-1$ is the $(q^s+1)$-th power of an element in $\F_{q^t}$ and by Proposition \ref{prop-W}, $m \in \mathscr{P}_s^+$ against the hypotheses. Similarly, if $t$ is odd and $q\equiv  1 \pmod 4$, then $-1$ is a square in $\F_q$ and $m \in \mathscr{P}_s^+$. 
In the case $t$ odd and $q\equiv 3 \pmod 4$, for every $z\in\F_{q^{2t}}^*$ such that $m=z^{1+q^s}$, we have $\Tr_{q^{2t}/q^t}(z)\neq 0$.  Indeed, let $$\alpha=\frac{M(x_1)}{x_2(h^{q^{s(t-1)}}+h^{q^{s}})}$$
and  $\mu^{q^s+1}=-1$ for some $\mu  \in \F_{q^{2}}\setminus \F_{q^t}$. Then $m=(\mu \alpha)^{q^{s}+1}$, and by Proposition \ref{prop-W} we have $\alpha \in \ker \Tr_{q^{2t}/q^t}$. Now, let $z \in \F_{q^{2t}}^*$ be such that $m=z^{q^s+1}$. Then there exists $\theta \in \F_{q^{2t}}^*$ such that $\theta^{q^s+1}=1$ and $z=\theta \mu \alpha$. Since $\gcd(q^s+1,q^{2t}-1)=q+1$, we get $\theta^{q+1}=1$. In particular $\theta\in\mathbb F_{q^2}$, and, since $t$ is odd,
$\theta^{q^t}=\theta^q$.
Similarly, $\mu^{q^t}=\mu^q$. Therefore,
$$\Tr_{q^{2t}/q^t}(z)= (\theta \mu - \theta^{q}\mu^{q}) \alpha.$$ Now suppose that $\Tr_{q^{2t}/q^t}(z)=0$, then $\theta \mu \in \F_{q}$. However, $(\theta\mu)^{q+1}=-1$ and if  $\theta\mu\in\mathbb F_q$, then this would imply $(\theta\mu)^2=-1$, getting that $-1$ is a square in $\F_q$. Since $q \equiv 3 \pmod 4$, this is not the case. Then,  $\Tr_{q^{2t}/q^t}(z)\neq 0$ and, hence, $m\notin \mathscr{P}_s^+$, a contradiction. 


Then, by \eqref{Fq-case}, we get 
\begin{equation}\label{eq:a^q-1}
\begin{split}
    a^{\qs-1}&=\frac{M(x_1)M(x_1)^\qs-h^{q^{st}-q^{s(t+1)}}x_2L_m(x_2)(1+h^{q^{s(t-1)}-\qs})}{mx_1^\qs(1+h^{q^{st}-q^{2s}})M(x_1)^\qs-h^{q^{st}-q^{2s}}L_m(x_2)L_m(x_2)^\qs}=\\
    &=\frac{h^{\qs-1}}{m} \cdot \frac{1+h^{q^{s(t-1)}-\qs}}{1+h^{q^{st}-q^{2s}}} \cdot \frac{x_1M(x_1)-x_2L_m(x_2)}{x_1^\qs M(x_1)^\qs-x_2^\qs L_m(x_2)^\qs}=\\
    &=\frac{1}{m}\cdot\Biggl( \frac{h}{(1+h^{q^{s(t-1)}-\qs})(x_1M(x_1)-x_2L_m(x_2))} \Biggr)^{\qs-1}.
\end{split}
\end{equation}
If $m$ is not a $(q^s-1)$-power we have a contradiction. So, let $z\in\mathbb{F}_{q^{2t}}^*$ such that $z^{q^s-1}=m$. From the hypothesis on $m$ we have that $z^{q^{st}}=\mu z$ with $\mu\in\mathbb{F}_{q}^*$ and $ \mu\neq -1$.
From \eqref{eq:a^q-1} we obtain
that there exists $\lambda \in \mathbb{F}_{q}^*$ such that 
\begin{equation*}
    a=\lambda \cdot \frac{h}{z(1+h^{q^{s(t-1)}-\qs})(x_1M(x_1)-x_2L_m(x_2))}.
\end{equation*}
Since $a \in \ker R$,
\begin{equation}\label{ainKerR}
\begin{aligned}
&\frac{h^{q^{st}}}{\mu z\left((1+h^{q^{s(t-1)}-\qs})(x_1M(x_1-x_2L_m(x_2))\right)^{q^{st}}}\\
&  + h^{q^{s(t-1)}-\qs}\cdot \frac{h}{z(1+h^{q^{s(t-1)}-\qs})(x_1M(x_1)-x_2L_m(x_2))}=0.
\end{aligned}
\end{equation}
Recalling that $x_1 \in \ker L_m$ and $x_2 \in \ker M$    we have that 
\begin{equation*}
(x_1M(x_1)-x_2L_m(x_2))^{q^{st}}=h^{q^{st}-1}(x_1M(x_1)-x_2L_m(x_2)) 
\end{equation*}
and hence 
\begin{equation}\label{x1M-x2L}
\begin{aligned}
&\frac{h^{q^{st}}}{\mu zh^{q^{st}-1}\left(1+h^{q^{s(t-1)}-\qs}\right )^{q^{st}} (x_1M(x_1)-x_2L_m(x_2)}\\
&+ h^{q^{s(t-1)}-\qs}\cdot \frac{h}{z(1+h^{q^{s(t-1)}-\qs})(x_1M(x_1)-x_2L_m(x_2))}=0.
\end{aligned}
\end{equation}

which means $\mu=-1$, a contradiction.
\medskip

\noindent- $\textit{Case 3.2}$. Let $\lambda_1\notin\F_q$ and let $a$ be a nonzero solution of System \eqref{complete-linear-system}. 
Since $L(x_2) \neq 0 \neq M(x_1)$, multiplying the first and the second equation of \eqref{complete-linear-system}  by $M(x_1)^{q^s}$ and by $L_m(x_2)$, respectively, we obtain \eqref{Fq-case}.
If this system admits more than one solution, then each $2\times 2$ minor of the associated matrix of \eqref{complete-linear-system} is zero. In particular Equations \eqref{D2} holds true, obtaining a contradiction as in the previous case.

Then, System \eqref{complete-linear-system} must admits a unique nonzero solution $(a,a^\qs)  \in \F_{q^{2t}}^2$. By \eqref{Fq-case}, we get

\begin{equation*}
\begin{split}
    a^{\qs-1}=\frac{M(x_1)M(x_1)^\qs-h^{q^{st}-q^{s(t+1)}}x_2L_m(x_2)(1+h^{q^{s(t-1)}-\qs})}{mx_1^\qs(1+h^{q^{st}-q^{2s}})M(x_1)^\qs-h^{q^{st}-q^{2s}}L_m(x_2)L_m(x_2)^\qs}.
\end{split}
\end{equation*}This is again Equation \eqref{eq:a^q-1}. 
Repeating the arguments as in \textit{Case 3.1}, we get a contradiction. \qedhere
\end{proof}

\subsection{Case \texorpdfstring{$\mathrm{N}_{q^{2t}/q^t}(h)=1$}{N(h) = 1}}

In this subsection, we will prove Theorem \ref{main} in the case of $h \in \F_{q^{2t}} \setminus \F_{q^t}$ with $\mathrm{N}_{q^{2t}/q^t}(h)=1$. 
First of all, note that if $q \equiv 1 \pmod 4$, putting $k=\lambda h $ such that $\lambda \in \F_{q}$ with $\lambda^2=-1$. Then, $\mathrm{N}_{q^{2t}/q^t}(k)=-1$ and $\psi_{m,k,s}=\psi_{m,h,s}$. Hence, by Theorem \ref{t:main}, $\psi_{m,h,s}$ is scattered. Therefore, consider the case $q \equiv 3 \pmod{4}$.

\begin{prop}\label{h2=-1}
    Let $t\geq 3$ odd and $q \equiv 3 \pmod{4}$. Let $h\in \F_{q^{2t}} \setminus \F_{q^t}$ such that $h^2= -1$. Then $\psi_{m,h,s}$ is not scattered for any $m\in\Fqt$ and $s$ a positive integer such that $\gcd(s,2t)=1$.
\end{prop}
\begin{proof}
    Fix $h\in\F_{q^{2t}}$ such that $h^2=-1$, then $$\psi_{m,h,s}=m(X^{q^s}-X^{q^{s(t+1)}})+X^{q^{s(t-1)}}-X^{q^{s(2t-1)}}=m(X-X^{q^{st}})^{q^s}+(X-X^{q^{st}})^{q^{s(t-1)}}.$$
    Hence, $\Fqt\subset\ker \psi_{m,h,s} $, so $\psi_{m,h,s}$ is not scattered.
\end{proof}

\begin{theorem}\label{norma=1}
	Let $t\geq 3$ and $q \equiv 3 \pmod{4}$ and $s$ be an integer such that $\gcd(s,2t)=1$. For each $h \in \F_{q^{2t}} \setminus \F_{q^t}$ such that $\mathrm{N}_{q^{2t}/q^t}(h)=1$ with $h^2 \neq -1$ and $m \in \mathbb{F}_{q^t} \setminus (\mathscr{P}_s^+ \cup \mathscr{P}_s^{-})$, the $q^s$-linearized polynomial
 \begin{equation*}
 \psi_{m,h,s}=m(X^{q^s}-h^{1-q^{s(t+1)}}X^{q^{s(t+1)}})+X^{q^{s(t-1)}}+h^{1-q^{s(2t-1)}}X^{q^{s(2t-1)}}
 \end{equation*}
 is scattered.
\end{theorem}

\begin{proof}
By using the same notation and arguing as in the first part of the proof of Theorem \ref{t:main} we obtain the system 
\begin{equation}\label{system1+}
\begin{cases}
L_m(ax_1)-aM(x_1)=(\lambda_2-\lambda^{q^s}_1)L_m(x_2)\\
b^{q^s}L_m(x_2)^{q^s}-M(bx_2)^{q^s}=(\lambda_2-\lambda_1^{q^s})M(x_1)^{q^s}.
\end{cases}
\end{equation}
Since $a=\mu_1 \rho$, $b=\mu_2\tau$ and $\tau=h^{q^{s(t-1)}-q^s}\rho$, by Lemma \ref{bases+split}, it follows
\begin{equation*}
\lambda_1+\mu_1\rho=\lambda_1+\mu_1\rho(1-h^{q^{s(t-1)}-q^s})+\mu_1\tau=\lambda_2+\mu_2\tau.\end{equation*}
Since $\{1,\tau\}$ is an $\F_{q^t}$-basis of $\F_{q^{2t}}$ and $\rho(1-h^{q^{s(t-1)}-q^{s}})$ belongs to $\F_{q^t}$, we get $\mu_1=\mu_2$ and $b=h^{q^{s(t-1)}-q^s}a$.  The case $a=0$ can be settled as in Theorem \ref{t:main}. If $a \neq 0$, we again split our discussion into three cases.\\
\noindent\textbf{Case 1.} $x_1=0$. This case is analogous to \textbf{Case 1.} in Theorem \ref{t:main}.

\noindent\textbf{Case 2.} $x_2=0$. This case is analogous to \textbf{Case 2.} in Theorem \ref{t:main}.

\noindent\textbf{Case 3.} $x_1,x_2 \neq 0$. 

Recall that $a \in \ker R$, $b=h^{q^{s(t-1)}-\qs}a$, $\lambda_2 = \lambda_1+(1-h^{q^{s(t-1)}-\qs})a$, $x_1 \in \ker L_m$ and $x_2 \in \ker M$. Then, by \eqref{system1+}, $a$ turns out to be a nonzero solution of the following linear system in the unknowns $a$ and $a^{q^s}$:

\begin{equation}\label{complete-linear-system+}
    \begin{cases}
mx_1^\qs(1+h^{q^{st}-q^{2s}})a^\qs-\left(M(x_1)+(1-h^{q^{s(t-1)}-\qs})L_m(x_2)\right)a=(\lambda_1-\lambda^\qs_1)L_m(x_2)\\
h^{q^{st}-q^{2s}}L_m(x_2)^\qs a^\qs+\left(x_2^{q^{st}}(1+h^{q^{s(t-1)}-\qs})-(1-h^{q^{s(t-1)}-\qs})M(x_1)^\qs\right)a=(\lambda_1-\lambda_1^\qs)M(x_1)^\qs.
\end{cases}
\end{equation}

\noindent\textit{- Case 3.1} First of all, suppose that $\lambda_1 \in \mathbb{F}_q$, then System \eqref{complete-linear-system+} becomes

\begin{equation}\label{eq:case3.1_star+}
    \begin{cases}
mx_1^\qs(1+h^{q^{st}-q^{2s}})a^\qs-\left(M(x_1)+(1-h^{q^{s(t-1)}-\qs})L_m(x_2)\right)a=0\\
h^{q^{st}-q^{2s}}L_m(x_2)^\qs a^\qs+\left(x_2^{q^{st}}(1+h^{q^{s(t-1)}-\qs})-(1-h^{q^{s(t-1)}-\qs})M(x_1)^\qs\right)a=0.
\end{cases}
\end{equation}
and since $a$ is a nonzero solution then
\begin{eqnarray}\label{determinant-incomplete+}
 mx_1^\qs(1+h^{q^{st}-q^{2s}})\left(x_2^{q^{st}}(1+h^{q^{s(t-1)}-\qs})-(1-h^{q^{s(t-1)}-\qs})M(x_1)^\qs\right)=\nonumber\\=
 -h^{q^{st}-q^{2s}}L_m(x_2)^\qs\left(M(x_1)+(1-h^{q^{s(t-1)}-\qs})L_m(x_2)\right).
\end{eqnarray}

Since $L_m(x_2) \neq 0 \neq M(x_1)$, by \eqref{eq:case3.1_star+}, we have
\begin{equation*}
\begin{aligned}
M(x_1)^\qs &\left(mx_1^\qs(1+h^{q^{st}-q^{2s}})a^\qs-M(x_1)a\right)\\
&=L_m(x_2)\left(h^{q^{st}-q^{2s}}L_m(x_2)^\qs a^\qs+x_2^{q^{st}}(1+h^{q^{s(t-1)}-\qs})a\right)
\end{aligned}
\end{equation*}
and hence 
\begin{align}\label{aqs}
\nonumber \left(mx_1^\qs M(x_1)^\qs(1+h^{q^{st}-q^{2s}}) - h^{q^{st}-q^{2s}}L_m(x_2)L_m(x_2)^\qs\right)a^\qs=\\
=\left(M(x_1)M(x_1)^\qs+x_2^{q^{st}}(1+h^{q^{s(t-1)}-\qs})L_m(x_2)\right)a. 
\end{align}

Next we want to show that the coefficient of $a^\qs$ in \eqref{aqs} cannot be $0$. By way of contradiction, suppose that
\begin{equation}\label{D2+}
	mx_1^\qs\left(1+h^{q^{st}-q^{2s}}\right)M(x_1)^\qs=h^{q^{st}-q^{2s}}L_m(x_2)L_m(x_2)^{q^s}.
\end{equation}

Then, also the coefficient of $a$ in \eqref{aqs} must be $0$, and similarly to  \eqref{eq:m(q+1)}, we get 

\begin{equation}\label{eq:m(q+1)+}
\begin{split}
    m&=\frac{M(x_1)^{q^s+1}}{(x_2^\qs)^{q^{s(t-1)}+1}(1+h^{1-q^{s(t+2)}})^{q^{s(t-1)}+1}}=\\
    &=\frac{M(x_1)^{q^s+1}}{-x_2^\qs h^{q^{st}-q^{s(t+1)}}x_2(1+h^{q^{s(t-1)}-q^{s}})(1+h^{1-q^{s(t+2)}})}=\\
    &=\frac{M(x_1)^{q^s+1}}{x_2^{\qs+1} h^{q^s-1}(1+h^{q^{s(t-1)}-q^{s}})(1+h^{-q^{s(t-1)}+q^{s}})^\qs}=\\
    &=\frac{M(x_1)^{q^s+1}}{x_2^{\qs+1}(1+h^{q^{s(t-1)}-q^{s}})^{1+q^s} h^{q^s-1}(h^{-q^{st}+q^{2s}})}=\\
    &=\frac{M(x_1)^{q^s+1}}{x_2^{\qs+1}(1+h^{q^{s(t-1)}-q^{s}})^{1+q^s}(h^\qs)^{1+\qs}}=\\
    &=\left(\frac{M(x_1)}{x_2(h^{q^s}+h^{q^{s(t-1)}})}\right)^{q^s+1}.
\end{split}
\end{equation}

By Proposition \ref{prop-W}, $m \in \mathscr{P}_s^+$, a contradiction. 
Finally, the result follows  as $\textit{Case 3.1}$ of Theorem \ref{t:main}.

\noindent- $\textit{Case 3.2}$. Analogous to the $\textit{Case 3.2}$ in Theorem \ref{t:main}.

\end{proof}

\section{Equivalence issue}\label{sec:equiv}

In this section, we will show that the family of quadrinomials in Theorem \ref{main} is not equivalent to the three known families of  scattered polynomials of $\F_{q^n}$ existing for infinitely many $n$ and $q$. 


First, we prove that the family of quadrinomials defined in Theorem~\ref{main} is not equivalent to the families of pseudoregulus-type and Lunardon–Polverino LP–type polynomials.  
We will use some techniques from the theory of linear sets. More precisely, we will show that the linear sets associated with pseudoregulus-type and LP-type polynomials are not equivalent to those associated with a polynomial $\psi_{m,h,s}$.
Consequently, $\psi_{m,h,s}$ is $\GaL$-equivalent neither to $f_{1,s}$ nor $f_{2,s}$.\\
To this end, we recall the following notions. 
In \cite[Theorem 2]{Lunardon_Polverino}, it has been shown that every linear set of rank $u$ of $\Lambda=\PG(r-1,q^n)$ spanning it over $\F_{q^n}$, is either a canonical subgeometry of $\Lambda$ or it can be obtained as projection of a canonical subgeometry $\Sigma \cong \PG(u-1,q)$ of $\Sigma^* =\PG(u-1,q^n) \supset \Lambda$ from a suitable $(u-r-1)$-dimensional subspace $\Gamma$ of $\Sigma^*$ onto $\Lambda$ such that $\Gamma \cap \Sigma = \emptyset = \Gamma \cap \Lambda$. The subspaces $\Gamma$ and $\Lambda$ are also called the {\it vertex} and {\it axis} of the projection, respectively. We denote such a projection by the symbol $\mathrm{p}_{\Gamma,\Lambda}(\Sigma)$.

Let $\sigma$ be a collineation of $\PG(n-1,q^n)$ fixing pointwise  a canonical subgeometry $\Sigma$ of $\PG(n-1,q^n)$. Let $\Gamma$ be a non-empty $k$-dimensional subspace of $\PG(n-1,q^n)$. In \cite{Zanella-Zullo}, the authors define the \textit{intersection number of $\Gamma$ with respect to $\sigma$}, denoted by $\mathrm{intn}_{\sigma}(\Gamma)$, as the least positive  integer $\gamma$ satisfying 
\begin{equation*}
    \dim(\Gamma \cap \Gamma^\sigma \cap \ldots \cap \Gamma^{\sigma^\gamma}) > k -2 \gamma.
\end{equation*}
It is easy to see that $\mathrm{intn}_{\sigma}(\Gamma)$ is invariant under the action of the automorphism group $\mathrm{Aut}(\Sigma)=\{ \varphi \in \mathrm{P\Gamma L}(n,q^n) : \varphi \sigma = \sigma \varphi \}$. In \cite[Theorem 2.3]{Csajbok_Zanella1} and \cite[Theorem 3.2]{Zanella-Zullo}, the authors give a characterization of the families of scattered  $\F_q$-linear sets of pseudoregulus type and LP-type by means of the intersection number of the vertex of the projection, see also \cite{equivalences}. More precisely, let $\Gamma$ be an $(n-3)$-dimensional subspace and let $\Lambda$ be a line of $\PG(n-1,q^n)$ such that $\Gamma \cap \Lambda=\emptyset$, then we have:
\begin{itemize}
\item $L=\mathrm{p}_{\Gamma, \Lambda}(\Sigma)$ is a scattered $\F_q$-linear set of pseudoregulus type if, and only if, there exists $\sigma$ a collineation of order $n$ fixing pointwise $\Sigma$  such that $\mathrm{intn}_{\sigma}(\Gamma)=1$; 
\item $L=\mathrm{p}_{\Gamma, \Lambda}(\Sigma)$ is a scattered $\F_q$-linear set of LP-type if, and only if,
\begin{itemize}
\item[a)] there exists $\sigma$ a collineation of order $n$ fixing pointwise $\Sigma$ such that $\mathrm{intn}_{\sigma}(\Gamma)=2$;
\item[b)] there exist a unique point $P \in \PG(n - 1,q^n)$ and some point $Q$ such that $$\Gamma = \langle P,P^{\sigma}, \ldots, P^{\sigma^{n-4}},Q \rangle;$$
\item[c)] the line $\langle P^{\sigma^{n-1}}, P^{\sigma^{n-3}} \rangle$ meets $\Gamma$.
\end{itemize}
\end{itemize}

The following proposition shows that the family of quadrinomials defined as in Theorem \ref{main} is not equivalent to the families of pseudoregulus type and LP-type.

\begin{prop}
   Let $\Sigma$ be a canonical subgeometry of $\PG(2t-1,q^{2t})$, $t \geq 5$, and  $\sigma$ denote a collineation of order $2t$ of $\PG(2t-1,q^{2t})$ fixing pointwise $\Sigma$.  Let  $\mathrm{p}_{\Gamma, \Lambda}(\Sigma)$ be a linear set of $\PG(2t-1,q^{2t})$ equivalent to $L_{\psi_{m,h,s}}$, with $\gcd(s,2t)=1$. Then $\mathrm{intn}_{\sigma}(\Gamma)\geq 3$.
\end{prop}

\begin{proof}
Consider $$\Sigma=\left \{ \langle (x,x^{q^s},\ldots,x^{q^{s(2t-1)}}) \rangle _{\F_{q^{2t}}} : x \in \F_{q^{2t}}^* \right \}$$ a canonical subgeometry of $\PG(2t-1,q^{2t})$. This is fixed by the collineation 
$$\sigma : \langle (x_0,x_s,\ldots,x_{s(2t-1)}) \rangle_{ \F_{q^{2t}}} \in \PG(2t-1,q^{2t}) \longrightarrow \langle(x_{s(2t-1)}^{q^{s}},x_{0}^{q^{s}},x_s^{q^s}, \ldots, x_{s(2t-2)}^{q^{s}}) \rangle _{\F_{q^{2t}}} \in \PG(2t-1,q^{2t}).$$ It is straightforward to see that the linear set $L_{\psi_{m,h,s}}$ is equivalent to $\mathrm{p}_{\Gamma, \Lambda}(\Sigma)$ where

\begin{equation*}
    \Gamma : \begin{cases} 
        X_{0}  = 0\\
      m(X_s-h^{1-q^{s(t+1)}}X_{s(t+1)})+X_{s(t-1)}+h^{1-q^{s(2t-1)}}X_{s(2t-1)}=0,
       \end{cases}
\end{equation*}
and $\Lambda$ is defined by the equations
\[
X_{sj}=0
\quad\text{for every positive integer } j \text{ with } 
j\not\equiv 0,t-1 \pmod{2t}.
\] 
Then,
    $$\Gamma^{\sigma} : \begin{cases} 
        X_{s}  = 0\\
      m^{q^{s}}(X_{2s}-h^{q^s-q^{s(t+2)}}X_{s(t+2)})+X_{st}+h^{q^s-1}X_0=0,
       \end{cases}
    $$
    and
     $$
    \Gamma^{\sigma^2} : \begin{cases} 
        X_{{2s}}  = 0\\
      m^{q^{2s}}(X_{3s}-h^{q^{2s}-q^{s(t+3)}}X_{s(t+3)})+X_{s(t+1)}+h^{q^{2s}-q^{s}}X_{s}=0.
       \end{cases}
    $$
    Therefore, it follows $\dim( \Gamma \cap  \Gamma^{\sigma})=2t-5$ and $\dim( \Gamma \cap  \Gamma^{\sigma} \cap \Gamma^{\sigma^2})=2t-7$. Hence. $\mathrm{intn}_{\sigma}(\Gamma) \geq 3$ and this concludes the proof.
\end{proof}

Next, we deal with possible equivalences between two polynomials $\psi_{m,h,s}$ and $\psi_{\mu,k,\ell}$. 

\begin{remark}\label{remequiv}
\textnormal{    Let $s,\ell$ be positive integers with $\gcd(s,2t)=1=\gcd(\ell,2t)$, $m,\mu\in\F_{q^t}$, $h,k\in\F_{q^{2t}}$ such that they respect the assumptions of Theorem \ref{main}, 
    then $\psi_{m,h,s}\sim_{\GaL} \psi_{\mu,k,\ell}$ if and only if there exist $\tau\in\mathrm{Aut}(\F_{q^{2t}})$ such that $\psi_{m,h,s}\sim_{\mathrm{GL}} \psi_{\mu^\tau,k^\tau,\ell}$}
\end{remark}

Since the analysis is considerably more involved, and to make the section more accessible, we will postpone the proof of the following theorem to the Appendix.

\begin{theorem}\label{equivalence}
Let $t \geq 5$, $s,\ell$ be positive integers with $\gcd(s,2t)=1=\gcd(\ell,2t)$, $m,\mu\in\F_{q^t}$, $h,k\in\F_{q^{2t}}$ verifying the assumptions of Theorem \ref{main} and let $\psi_{m,h,s},\psi_{\mu,k,\ell}$ be defined as above. Then we have the following results:
\begin{itemize}
    \item[a)] if $  \ell \not\equiv \pm s \pmod{2t}$ and $  \ell \not\equiv t\pm s \pmod{2t}$, then $\psi_{m,h,s}\not\sim_{\mathrm{GL}} \psi_{\mu,k,\ell}$;
    \item[b)] if $  \ell \equiv - s \pmod{2t}$ and $\psi_{m,h,s}\sim_{\mathrm{GL}} \psi_{\mu,k,\ell}$, then we have that $kh\in\mathbb{F}_{q^{\gcd(t-2,2t)}}$ and we have one of the two following properties on $m,\mu$: \begin{itemize}
        \item there exist $z\in\F_{q^{2t}}$ such that $z^{q^{st}}=(kh)^{q^{3s}-1}z$ and $m\mu^\qs=z^{q^{s(t-2)}-1}$;
        \item there exist $z\in\F_{q^{2t}}$ such that $z^{q^{st}}=-(kh)^{q^{3s}-1}z$ and $m\mu=-z^{q^{s(t-2)}-1}$.
    \end{itemize}
    \item[c)] if $  \ell \equiv  s \pmod{2t}$ and $\psi_{m,h,s}\sim_{\mathrm{GL}} \psi_{\mu,k,\ell}$, then we have that $h/k\in\mathbb{F}_{q^{\gcd(s(t-2),2t)}}$ and we have one of the two following properties on $m,\mu$: \begin{itemize}
        \item there exist $z\in\F_{q^{2t}}$ such that $z^{q^{st}}=-(h/k)^{q^{3s}-1}z$ and $m\mu^{-q^{s(t-1)}}=-z^{q^{s(t-2)}-1}$;
        \item there exist $z\in\F_{q^{2t}}$ such that $z^{q^{st}}=(h/k)^{q^{3s}-1}z$ and $m/\mu=z^{q^{s(t-2)}-1}$.
    \end{itemize}
    \item[d)] if $  \ell \equiv t- s \pmod{2t}$ and $\psi_{m,h,s}\sim_{\mathrm{GL}} \psi_{\mu,k,\ell}$, then we have that $h/k\in\mathbb{F}_{q^{\gcd(s(t-2),2t)}}$ and we have one of the two following properties on $m,\mu$: \begin{itemize}
        \item there exist $z\in\F_{q^{2t}}$ such that $z^{q^{st}}=(h/k)^{q^{3s}-1}z$ and $m\mu^\qs=-z^{q^{s(t-2)}-1}$;
        \item there exist $z\in\F_{q^{2t}}$ such that $z^{q^{st}}=-(h/k)^{q^{3s}-1}z$ and $m\mu=z^{q^{s(t-2)}-1}$.
\end{itemize}
\item[e)] if $  \ell \equiv t+ s \pmod{2t}$ and $\psi_{m,h,s}\sim_{\mathrm{GL}} \psi_{\mu,k,\ell}$, then we have that $hk\in\mathbb{F}_{q^{\gcd(s(t-2),2t)}}$ and we have one of the two following properties on $m,\mu$: \begin{itemize}
        \item there exist $z\in\F_{q^{2t}}$ such that $z^{q^{st}}=-(hk)^{q^{3s}-1}z$ and $m\mu^{-q^{s(t-1)}}=z^{q^{s(t-2)}-1}$;
        \item there exist $z\in\F_{q^{2t}}$ such that $z^{q^{st}}=(hk)^{q^{3s}-1}z$ and $m/\mu=-z^{q^{s(t-2)}-1}$.
\end{itemize}
\end{itemize}
\end{theorem}

The following corollary of Theorem \ref{equivalence} ensures that the family of polynomials in Theorem \ref{main} contains new examples of scattered polynomials other than those already investigated in \cite{longobardi_marino_trombetti_zhou, SmaZaZu}. 
\begin{cor}
    Let $q$ be an odd prime power, $t\geq5$ a positive integer and $\gcd(s,2t)=1$. If $t$ is odd and $q\geq7$ or $t$ is even and $q\geq5$ there exists a pair $(m,h)$ (verifying the assumptions of Theorem \ref{main}) such that for any  positive integer $\ell$ with $\gcd(\ell,2t)=1$, for any $k\in\F_{q^{2t}}$ with $\mathrm{N}_{q^{2t}/q^t}(k)=-1$, and for any $\mu\in\F_{q^t}$ we have $$\psi_{1,k,\ell}\not\sim_{\GaL}\psi_{m,h,s}\not\sim_{\GaL}\psi_{\mu,1,\ell}.$$
\end{cor}
\begin{proof}
    By Remark \ref{remequiv}, it is enough to prove that $$\psi_{1,k,\ell}\not\sim_{\mathrm{GL}}\psi_{m,h,s}\not\sim_{\mathrm{GL}}\psi_{\mu,1,\ell},$$
    for any $k\in\F_{q^{2t}}$ with $\mathrm{N}_{q^{2t}/q^t}(k)=-1$, and for any $\mu\in\F_{q^t}$.
    Note that to obtain $\psi_{m,h,s}\not\sim_{\mathrm{GL}}\psi_{\mu,1,\ell}$ we only need that $h\notin\mathbb{F}_{q^{\gcd(t-2,2t)}}$. 
By Theorem \ref{equivalence}    To have $\psi_{1,k,\ell}\not\sim_{G L}\psi_{m,h,s}$  we need that do not exist $k,z\in\F_{q^{2t}}$ with $\mathrm{N}_{q^{2t}/q^t}(k)=-1$ such that $m=\pm z^{q^{s(t-2)}-1}$ with $z^{q^{st}}=\pm(\xi)^{q^{3s}-1}z$, where $\xi=hk^{\pm1}\in\mathbb{F}_{q^{\gcd(t-2,2t)}}.$

Firstly, note that in every case we will find $h$ with the assumptions of Theorem \ref{t:main} or Theorem \ref{norma=1} such that $h\notin\mathbb{F}_{q^{\gcd(t-2,2t)}}$. Now, since the set of $(q-1)$-th power of an element in $\F_{q^{2t}}^*$ is closed by the multiplication of $-1$, it is sufficient to take  $m$ such that is not a $(q-1)$-th power of an element in $\F_{q^{2t}}^*$ .
Let us compute the cardinality of $\mathscr{D}=\{x^{q-1}\,:\,x \in \F_{q^{2t}}^*\}\cap\F_{q^t}^*$. Given $\alpha$ a generator of $\F_{q^{2t}}^*$ and consider the multiplicative subgroup of $\F_{q^{2t}}^*$,  $\langle\alpha^{q-1}\rangle$ and $\langle\alpha^{q^t+1}\rangle$.
Then,
$\langle\alpha^{q-1}\rangle\cap\langle\alpha^{q^t+1}\rangle=\langle\alpha^{\text{lcm}(q-1,q^t+1)}\rangle$ and its size is $$\frac{q^{2t}-1}{{\text{lcm}}(q-1,q^t+1)}=\frac{(q^t-1)\gcd(q-1,q^t+1)}{q-1}=2\frac{q^t-1}{q-1}.$$

If $t$ is even, we need $m\in\F_{q^t}^*$ such that is not in $\mathscr{P}^+_s\cup \mathscr{D}$. By \cite[Proposition 2.5]{SmaZaZu} and Lemma \ref{lem:P}, we have that $\vert \mathscr{P}^+_s \vert =\frac{q^t-1}{q-1}$, so when $q\geq5$ we obtain 
$$\vert\F_{q^t}^* \vert =q^t-1>2\frac{q^t-1}{q-1}+\frac{q^t-1}{q-1}\vert= \vert \mathscr{P}^+_s\cup \mathscr{D} \vert  \geq  \vert \mathscr{P}^+_s \cup \mathscr{P}^-_s. \vert $$
If $t$ is odd and $q\equiv3\pmod{4}$, we need $m\in\mathscr{P}^+_s$ such that is not in $\mathscr{D}$. By \cite[Proposition 2.5]{SmaZaZu} and Lemma \ref{lem:P}, we have that $\vert \mathscr{P}^+_s \vert =\frac{q^t-1}{2}$. So when $q\geq7$, we obtain $$\vert \mathscr{P}^+_s \vert =\frac{q^t-1}{2}>2\frac{q^t-1}{q-1}=\vert \mathscr{D} \vert .$$
In this case, we must choose $h$ with $\mathrm{N}_{q^{2t}/q^t}(h)=-1$.
Finally, if $t$ is odd and $q \equiv 1 \pmod 4$, we need again $m\in\F_{q^t}^*$ such that is not in $\mathscr{P}^+_s\cup \mathscr{D}$. Again, by \cite[Proposition 2.5]{SmaZaZu} and Lemma \ref{lem:P},  $\vert \mathscr{P}^+_s \vert =\frac{q^t-1}{2}$, Then, for $q\geq7$, we obtain $$\vert \F_{q^t}^* \vert =q^t-1>2\frac{q^t-1}{q-1}+\frac{q^t-1}{2}\geq\ \vert \mathscr{P}^+_s \cup \mathscr{D}  \vert \geq \vert \mathscr{P}^+_s \cup \mathscr{P}^-_s \vert .$$
Therefore, in every case, there exists such a suitable pair $(m,h) \in \F_{q^t} \times \F_{q^{2t}}$ and this completes the proof.  
\end{proof}

\section{The stabilizer of \texorpdfstring{$U_{\psi_{m,h,s}}$}{U_{psi}}}
\label{sec:stab}

The aim of this section is to compute the stabilizer in $\mathrm{GL}(2,q^n)$ of the subspace $U_{\psi_{m,h,s}}$. From now on, we will denote by $\mathcal{S}_{n,q,s}$ the set of all scattered polynomials $f \in \tilde{\mathcal{L}}_{n,q,s}$
such that $G_f^\circ$ is not isomorphic to $\mathbb{F}_q$.

\begin{definition}\label{standard-form}
\textnormal{For a scattered polynomial $F=\sum_{i=0}^{n-1}b_iX^{q^{si}}$, let 
\[\Delta_F=\{(i-j) \mod n\colon b_ib_j\neq0 \, \text{and} \, i \neq j \}\cup\{n\},\] and $r_F$ be the greatest common divisor
of $\Delta_F$.
If $r_F>1$ then $F$ is in \emph{standard form}.}
\end{definition}
A scattered polynomial $F$ in standard form has the following shape:
\begin{equation}\label{standardshape}
\sum^{n/r
-1}_{
j=0}
b_jX^{q^{jr+s}}
,
\end{equation}
where $r = r_F$ and $1 \leq  s < r$ with $\gcd(s,r)=1$, see \cite[Definition 2.4]{survey}.
Let $f$ be  a scattered polynomial in $\mathcal{S}_{n,q,s}$. 
Then, $f$ is $\GL$-equivalent to a polynomial $F$ in standard form, see \cite[Theorem 2.5]{survey}.

\begin{theorem}\label{eqRI}
 Let $\cC_F = \langle X, F \rangle_{\F_{q^n}}  \subseteq \tilde{\mathcal{L}}_{n,q,s}$ be a linear  MRD code with minimum distance $n-1$. The following statements are equivalent: 
 \begin{itemize}
     \item [$(i)$] $|G^\circ_ F| = q^r$, $r > 1$, and all elements of $G^\circ _F$ are diagonal.
    \item  [$(ii)$] the polynomial $F$ is in standard form. 
    \item [$(iii)$] the right idealizer $$I_R(\cC_F)=\{\alpha X : \alpha \in \F_{q^r}\}.$$ 
    \end{itemize}
    Moreover, if the above conditions $(i)$, $(ii)$ and $(iii)$ hold, then $r = r_F$ and $G^\circ_F =
   \left \{ \begin{pmatrix}
         \alpha & 0\\
         0 & \alpha^{q^{s}}
    \end{pmatrix}  : \alpha \in \F_{q^r} \right \}$, with $\gcd(s, r) =1$.
\end{theorem}


\begin{theorem}
Let $t\geq 5$, and $\psi_{m,h,s}$ with $(m,h) \in \F_{q^t} \times \F_{q^{2t}}$ as in Theorem \ref{main}. Then,
\begin{itemize}
\item if $t$ is even, then
$$G_{\psi_{m,h,s}}^{\circ}= \left \{ \begin{pmatrix}
    \alpha & 0 \\
     0 & \alpha^{q^{s}}
\end{pmatrix} \colon \alpha \in \F_{q^2} \right \};$$
\item if $t$ is odd, then

\begin{equation*}
\begin{aligned}
    G_{\psi_{m,h,s}}^{\circ}=\Biggl \{ &\begin{pmatrix} \alpha & \xi\frac{m^R}{h^{q^s}+h^{q^{s(t-1)}}} \\
   - \xi \mathrm{N}_{q^{2t}/q^t}(h)\left ( m^{Rq^s+1}  h^{q^s} 
    + m^{q^{s(t-1)}(R+1)}h^{q^{s(t-1)}} \right )  & \alpha \end{pmatrix} \colon \\ 
    & \alpha \in \F_{q}, \,\, \xi \in \F_{q^2}  \textnormal{ s.t. } \xi^{q^s}+\xi=0  \Biggr \}
    \end{aligned}
    \end{equation*}
where $R=-\frac{q^{s(t+1)-1}}{q^{2s}-1}$.
\end{itemize}
In particular $I_R(\cC_{\psi_{m,h,s}})$ is isomorphic to $\F_{q^2}$.
\end{theorem}

\begin{proof}
First, assume $t$ even and $\psi:=\psi_{m,h,s}$. It is straightforward to see $$\Delta_{\psi}=\{2,t-2,t,2t-2,2t\}.$$
Thus $\psi$ is in standard form  and $r_{\psi}=2$. By Theorem \ref{eqRI}, we get easily the statement.\\
Now, assume $t$ odd.  Let $A=\begin{pmatrix}  \alpha & \beta \\ \gamma & \delta\end{pmatrix} \in G_{\psi}^\circ$. Then, for any $ x \in \F_{q^{2t}}$ there exists $y \in \F_{q^{2t}}$  such that 
    $$
    A \begin{pmatrix}
        x \\ \psi_{m,h,s}(x)
    \end{pmatrix} =\begin{pmatrix}
        y \\ \psi_{m,h,s}(y)
    \end{pmatrix}.
    $$
    
    Hence, $\gamma x+\delta\psi_{m,h,s}(x)=\psi_{m,h,s}(\alpha x+\beta \psi_{m,h,s}(x))$ for any $x \in \F_{q^{2t}}$ and this leads to the following equation in $\tilde{\mathcal{L}}_{2t,q,s}$
\begin{equation}\label{eq:equivalence}
\gamma X + \delta \psi_{m,h,s}(X)=\psi_{m,h,s}(\alpha X +\beta \psi_{m,h,s}(X)).
\end{equation} 
  Equating the coefficients with the same $q^s$-degree in  \eqref{eq:equivalence} and recalling that $m \in \F_{q^t}$, we get the following conditions:
    \begin{equation}\label{coeff:$x$}
        \gamma=m\beta^{q^s}h^{q^s-1}-mh^{1-q^{s(t+1)}}\beta^{q^{s(t+1)}}-m^{q^{s(t-1)}}\beta^{q^{s(t-1)}}h^{q^{s(t-1)}-1}+h^{1-q^{s(2t-1)}}m^{q^{s(t-1)}}\beta^{q^{s(2t-1)}};
    \end{equation}
\begin{equation}\label{coeff:$x^{q^s}$}
        \delta m=m\alpha^{q^s};
    \end{equation}
    \begin{equation}\label{coeff:$x^{q^{2s}}$}  m^{q^s+1}\beta^{q^s}+m^{q^s+1}\beta^{q^{s(t+1)}}h^{1-q^{2s}}=0;
    \end{equation}
    \begin{equation}\label{coeff:$x^{q^{s(t-2)}}$}  \beta^{q^{s(t-1)}}h^{q^{s(t-1)}-q^{s(t-2)}}+h^{1-q^{s(2t-1)}}\beta^{q^{s(2t-1)}}=0;
    \end{equation}
    \begin{equation}\label{coeff:$x^{q^{s(t-1)}}$}  \delta=\alpha^{q^{s(t-1)}};
    \end{equation}
    \begin{equation}\label{coeff:$x^{q^{st}}$} m\beta^{q^s}-m\beta^{q^{s(t+1)}}h^{1-q^{st}}+m^{q^{s(t-1)}}\beta^{q^{s(t-1)}}-m^{q^{s(t-1)}}\beta^{q^{s(2t-1)}}h^{1-q^{st}}=0;
    \end{equation}
    \begin{equation}\label{coeff:$x^{q^{s(t+1)}}$}
-\delta mh^{1-q^{s(t+1)}}=-mh^{1-q^{s(t+1)}}\alpha^{q^{s(t+1)}};
    \end{equation}
    \begin{equation}\label{coeff:$x^{q^{s(t+2)}}$}
-m^{q^s+1}\beta^{q^s}h^{q^s-q^{s(t+2)}}-h^{1-q^{s(t+1)}}m^{q^s+1}\beta^{q^{s(t+1)}}=0;
    \end{equation}
    \begin{equation}\label{coeff:$x^{q^{s(2t-2)}}$}
\beta^{q^{s(t-1)}}+\beta^{q^{s(2t-1)}}h^{1-q^{s(2t-2)}}=0;
    \end{equation}

    \begin{equation}\label{coeff:$x^{q^{s(2t-1)}}$}
\delta h^{1-q^{s(2t-1)}}=h^{1-q^{s(2t-1)}}\alpha^{q^{s(2t-1)}}.
    \end{equation}
    By \eqref{coeff:$x^{q^s}$}, \eqref{coeff:$x^{q^{s(t-1)}}$}, \eqref{coeff:$x^{q^{s(t+1)}}$}, \eqref{coeff:$x^{q^{s(2t-1)}}$}, we have $\alpha \in \F_q$ and $\delta=\alpha$. Note that \eqref{coeff:$x^{q^{2s}}$},  \eqref{coeff:$x^{q^{s(t-2)}}$}, \eqref{coeff:$x^{q^{s(t+2)}}$} and  \eqref{coeff:$x^{q^{s(2t-2)}}$} are  equivalent under the hypothesis on $h$. Putting together these conditions with \eqref{coeff:$x^{q^{st}}$} we get
    \begin{equation}\label{bq2}
    m^{q^s}\beta^{q^{2s}}(1+h^{q^{3s}-q^{s(t+1)}})=m\beta h^{q^s-q^{s(2t-1)}}(1+h^{q^{s(2t-1)}-q^{s(t+1)}}).
    \end{equation}
By Proposition \ref{p:hcondition}, we have $1+h^{q^{3s}-q^{s(t+1)}} \ne 0$ and $1+h^{q^{s(2t-1)}-q^{s(t+1)}} \ne 0$. The roots of equation above in the unknown $\beta$ determine the kernel of an $\F_{q^2}$-linear map, then it has one or $q^2$ solutions. By \eqref{coeff:$x^{q^{s(2t-2)}}$}, since $\beta^{q^{st}}+\beta h^{q^s-q^{s(2t-1)}}=0$, it follows that 
the number of allowable values $\beta$ in \eqref{bq2} is either one or $q$. Let 
$$z=\frac{1}{h^{q^s}+h^{q^{s(t-1)}}}$$ and put $R:=-\frac{q^{s(t+1)}-1}{q^{2s}-1}$. Clearly, since $t$ is odd, $R$ is an integer. Then, it is easy to check that the solutions of Equation \eqref{bq2} are of the shape $\beta=\xi z m^R$ for $\xi \in \F_{q^2}$ such that $\xi^{q^s}+\xi=0$.
Finally, by \eqref{coeff:$x$}, we get
\begin{equation*}
\begin{aligned}
            \gamma&=m\beta^{q^s}h^{q^s-1}-mh^{1-q^{s(t+1)}}\beta^{q^{s(t+1)}}-m^{q^{s(t-1)}}\beta^{q^{s(t-1)}}h^{q^{s(t-1)}-1}+h^{1-q^{s(2t-1)}}m^{q^{s(t-1)}}\beta^{q^{s(2t-1)}}\\
            &=-\xi \Bigl (
 z^{q^s} m^{R q^s+1} h^{q^s-1}
+ z^{q^{s(t-1)}} m^{q^{s(t-1)(R+1)}} h^{q^{s(t-1)}-1} \\
& \quad + z^{q^{s(t+1)}} m^{1+R q^s} h^{1-q^{s(t+1)}}
+ z^{q^{s(2t-1)}} m^{q^{s(t-1)}(R+1)} h^{1-q^{s(2t-1)}}
\Bigr )\\
            &=-\xi \left( m^{Rq^s+1} \left (\frac{h^{q^s-1}}{h^{q^{2s}}+h^{q^{st}}}+\frac{h^{1-q^{s(t+1)}}}{h^{q^{s(t+2)}}+h} \right)+m^{q^{s(t-1)}(R+1)} \left ( \frac{h^{q^{s(t-1)}-1}}{h^{q^{st}}+h^{q^{s(2t-2)}}} + \frac{h^{1-q^{s(2t-1)}}}{h+h^{q^{s(t-2)}}}\right) \right )\\
&=-\xi \Biggl ( m^{Rq^s+1}  h^{q^s} \left (\frac{1}{h^{q^{2s}+1}+N}+\frac{1}{h^{q^{s(t+2)}+q^{st}}+N} \right)\\
& \quad + m^{q^{s(t-1)}(R+1)}h^{q^{s(t-1)}} \left ( \frac{1}{N+h^{q^{s(2t-2)}+1}} + \frac{1}{N+h^{q^{s(t-2)}+q^{st}}}\right) \Biggr ),
\end{aligned}
\end{equation*}
where  
$N = \mathrm{N}_{q^{2t}/q^t}(h)$. Then,
\begin{equation*}
    \begin{aligned}
    \gamma & = -\xi \Biggl ( m^{Rq^s+1}  h^{q^s} \left (\frac{1}{h^{q^{2s}+1}+N}+\frac{h^{q^{2s}+1}}{1+N h^{q^{2s}+1}} \right)\\
    & \quad + m^{q^{s(t-1)}(R+1)}h^{q^{s(t-1)}} \left ( \frac{h^{q^{s(t-2)}+q^{st}}}{N h^{q^{s(t-2)}+q^{st}}+1} + \frac{1}{N+h^{q^{s(t-2)}+q^{st}}}\right) \Biggr )\\
&= 
   - \xi \mathrm{N}_{q^{2t}/q^t}(h)\left ( m^{Rq^s+1}  h^{q^s} 
    + m^{q^{s(t-1)}(R+1)}h^{q^{s(t-1)}} \right ).
\end{aligned}
\end{equation*}
Hence, $\vert I_R(\cC_{\psi_{m,h,s}}) \vert =q^2$. Since $\psi_{m,h,s}$ is scattered,  $I_R(\cC_{\psi_{m,h,s}})$ is a subfield of $\F_{q^{2t}}$, so necessarily $I_R(\cC_{\psi_{m,h,s}})$ is isomorphic to $\F_{q^{2}}$.

\end{proof}

\section{Necessary conditions}\label{sec:necessarie}

In this section we provide some partial results on the necessary conditions for the quadrinomial $\psi_{m,h,s}$ to be scattered. 

\begin{theorem}\label{q-1}
    If $m \in \mathscr{P}^-_{s}$ and $h \in \F_{q^t}$ with $\mathrm{N}_{q^{2t}/q^t}(h)=\pm{1}$, then $\psi_{m,h,s}$ is not scattered.
\end{theorem}

\begin{proof}
By Remark \ref{hfqt}, it is enough to consider 
$$
\psi_{m,s}:=m(X^{q^s}-X^{q^{s(t+1)}})+X^{q^{s(t-1)}}+X^{q^{s(2t-1)}}.
$$
Then, let the equation \begin{equation}\label{eqscatt}
 \psi^{q^s}_{m,s}(x)y^{q^s}-\psi^{q^s}_{m,s}(y)x^{q^s}=0,        \end{equation} 
and let $m=\gamma^{q^s-1}$ with $\gamma\in \ker \mathrm{Tr}_{q^{2t}/q^t}$ .
    Since $\mathbb{F}_{q^{2t}}=\mathbb{F}_{q^t}\oplus \ker \mathrm{Tr}_{q^{2t}/q^t}$ , we know that $x=x_0+x_1$, $y=y_0+y_1$, with $x_0,y_0\in\mathbb{F}_{q^t}$ and $x_1,y_1 \in \ker \mathrm{Tr}_{q^{2t}/q^t}$. Then we have

    \begin{eqnarray*}
        \psi_{m,s}^{q^s}(x_0+x_1) &=& m^{q^s}\Big((x_0+x_1)^{q^{2s}}-(x_0+x_1)^{q^{s(t+2)}}\Big)+ (x_0+x_1)^{q^{st}}+(x_0+x_1)\\ &=& m^{q^s}\Big((x_0+x_1)^{q^{2s}}-(x_0-x_1)^{q^{2s}}\Big)+x_0-x_1+x_0+x_1  \\ &=&  2x_1^{q^{2s}}m^{q^s} + 2x_0.
    \end{eqnarray*}

    Hence, Equation \eqref{eqscatt} reads $A+B=0$, where

    $$A:=x_0y_0^{q^s}+m^{q^s}x_1^{q^{2s}}y_1^{q^s}-y_0x_0^{q^s}-m^{q^s}y_1^{q^{2s}}x_1^{q^s},$$ and $$B:=x_0y_1^{q^s}+m^{q^s}x_1^{q^{2s}}y_0^{q^s}-y_0x_1^{q^s}-m^{q^s}y_1^{q^{2s}}x_0^{q^s}.$$ Note that $m^{q^s}=(\gamma^{q^s})^{q^s-1}$ and $\gamma^{q^s}\in \ker \mathrm{Tr}_{q^{2t}/q^t}$. By Lemma \ref{power}, we have $m^{q^s}\in\mathbb{F}_{q^t}$ and $x_1^{q^s},x_1^{q^{2s}},y_1^{q^s},y_1^{q^{2s}}\in \ker \mathrm{Tr}_{q^{2t}/q^t}$. This implies $A\in\mathbb{F}_{q^t} $ and $B\in \ker \mathrm{Tr}_{q^{2t}/q^t}$, and hence from $\mathbb{F}_{q^t}\cap \ker \mathrm{Tr}_{q^{2t}/q^t}=\{0\}$, we obtain
    \begin{equation}\label{sistemaNec}
        \begin{cases}
         x_0y_0^{q^s}+m^{q^s}x_1^{q^{2s}}y_1^{q^s}-y_0x_0^{q^s}-m^{q^s}y_1^{q^{2s}}x_1^{q^s}=0\\
          x_0y_1^{q^s}+m^{q^s}x_1^{q^{2s}}y_0^{q^s}-y_0x_1^{q^s}-m^{q^s}y_1^{q^{2s}}x_0^{q^s}=0.
        \end{cases}
    \end{equation}
In particular, for $x_0=y_0$, \eqref{sistemaNec} reads
\begin{equation}\label{sistemNec2}
        \begin{cases}
          x_1^{q^{2s}}y_1^{q^s}-y_1^{q^{2s}}x_1^{q^s}=0\\
          x_0y_1^{q^s}+m^{q^s}x_1^{q^{2s}}x_0^{q^s}-x_0x_1^{q^s}-m^{q^s}y_1^{q^{2s}}x_0^{q^s}=0,
        \end{cases}
    \end{equation}    
    whence $y_1=\xi x_1$, with $\xi\in\mathbb{F}_{q}$, must hold. By substituting $y_1=\xi x_1$ in the second equation of \eqref{sistemNec2} we have

    \begin{equation*}
         x_0\xi x_1^{q^s}+m^{q^s}x_1^{q^{2s}}x_0^{q^s}-x_0x_1^{q^s}-m^{q^s}\xi x_1^{q^{2s}}x_0^{q^s}=(\xi-1)(x_0x_1^{q^s}-m^{q^s}x_1^{q^{2s}}x_0^{q^s})=0.
    \end{equation*}
    
    We can choose $x_1=\gamma^{-1}x_0^{-q^{s(2t-1)}}$ and $y_1=\xi x_1$ with $\xi\neq 1$, to have the needed solution.
Therefore, $x=x_0+\gamma^{-1}x_0^{-q^{s(2t-1)}}$ and $y=x_0+\xi\gamma^{-1}x_0^{-q^{s(2t-1)}}$ are solutions of \eqref{eqscatt} such that $x/y\not\in\mathbb{F}_{q}$, which proves that $\psi_{m,s}$ is not scattered.
\end{proof}

\begin{remark}
\textnormal{Theorem \ref{q-1}, together with \cite[Proposition 2.4]{SmaZaZu}, show that the sufficient conditions provided in \cite{SmaZaZu} are also necessary.}
\end{remark}

Based on computational results, the conditions on $m$ and $h$ in Theorem~\ref{main} appear to be necessary for the scatteredness of $\psi_{m,h,s}$. We conclude this section with the following open problem.

\begin{conjecture}\label{conj}
\textnormal{Let $t \geq 3$, $1 \leq s \leq 2t-1$ and consider the $q^s$-polynomial
$$\psi_{m,h,s}=m(X^{q^s}-h^{1-q^{s(t-1)}}X^{q^{s{t-1}}})+X^{q^{s(t+1)}}+h^{1-q^{s(2t-1)}}X^{q^{s(2t-1)}} \in \tilde{\mathcal{L}}_{n,q,s},$$
with $(m, h ) \in \F_{q^t} \times \F_{q^{2t}}$. Then, $\psi_{m,h,s}$ is scattered if and only if $m$ and $h$ satisfy the assumptions in Theorem~\ref{main}}.
\end{conjecture}

\section*{Acknowledgement}
The research  was supported by the Italian National Group for Algebraic and Geometric Structures and their Applications (INdAM -- GNSAGA Project CUP E53C24001950001). 

\vspace{1cm}
\noindent Dipartimento di Matematica e Applicazioni “Renato Caccioppoli”\\
Università degli Studi di Napoli Federico II,\\
Via Cintia, Monte S. Angelo I-80126 Napoli, Italy. \\
Email addresses: \\
\texttt{\{alessandro.giannoni, giovanni.longobardi\}@unina.it}

\vspace{1cm}
\noindent Dipartimento di Matematica e Informatica\\
Università di Perugia,\\
Via Vanvitelli 1, 06123 Perugia, Italy. \\
Email address: \\
\texttt{\{giovannigiuseppe.grimaldi, marco.timpanella\}@unipg.it}

\newpage

\section*{Appendix }

In this section, we provide the proof of Theorem \ref{equivalence}.

\begin{proof}
Let suppose that there exists an invertible matrix
\[
M=\begin{pmatrix}\alpha & \beta \\ \gamma & \delta\end{pmatrix}\in\mathrm{GL}(2, q^{2t})
\]
such that, for any $x \in \F_{q^{2t}}$,
\[
M\begin{pmatrix}x\\[4pt]\psi_{\mu,k,\ell}(x)\end{pmatrix}
=
\begin{pmatrix}y\\[4pt]\psi_{m,h,s}(y)\end{pmatrix}
.
\]
where   $y \in \F_{q^{2t}}$. 
Then, we get  $y= \alpha x + \beta\,\psi_{\mu,k,\ell}(x)$ and 
$\psi_{m,h,s}(y) = \gamma x + \delta\,\psi_{\mu,k,\ell}(x)$,
and hence for each $x \in \F_{q^{2t}}$,
\begin{equation}\label{polequiv}
\psi_{m,h,s}(\alpha x + \beta \psi_{\mu,k,\ell}(x)) - \gamma x - \delta \psi_{\mu,k,\ell}(x)=0.   
\end{equation}

We have
\begin{equation*}
\begin{aligned}
\psi_{m,h,s}(\alpha x + \beta \psi_{\mu,k,\ell}(x)) 
&=
m\Bigl(\alpha^{q^s} x^{q^s} + \beta^{q^s} \psi_{\mu,k,\ell}(x)^{q^s}
- h^{1-q^{s(t+1)}}(\alpha^{q^{s(t+1)}}x^{q^{s(t+1)}} \\
&+ \beta^{q^{s(t+1)}} \psi_{\mu,k,\ell}(x)^{q^{s(t+1)}})\Bigr)+ \alpha^{q^{s(t-1)}}x^{q^{s(t-1)}} + \beta^{q^{s(t-1)}} \psi_{\mu,k,\ell}(x)^{q^{s(t-1)}}\\
&+ h^{1-q^{s(2t-1)}}(\alpha^{q^{s(2t-1)}}x^{q^{s(2t-1)}} + \beta^{q^{s(2t-1)}} \psi_{\mu,k,\ell}(x)^{q^{s(2t-1)}})
\end{aligned}    
\end{equation*}
and each term $\psi_{\mu,k,\ell}(x)^{q^e}$ can be expanded as

\begin{equation*}
\psi_{\mu,k,\ell}(x)^{q^e} = \mu^{q^e}(x^{q^{\ell+e}} - k^{q^e(1-q^{\ell(t+1)})} x^{q^{\ell(t+1)+e}}) 
+ x^{q^{\ell(t-1)+e}} + k^{q^e(1-q^{\ell(2t-1)})} x^{q^{\ell(2t-1)+e}},
\end{equation*}
with all exponents taken modulo $2t$.



Let 
\[
R = \Bigl\{\, e + f \;\bmod 2t \;\Big|\; e \in \{0, s, s(t+1), s(t-1), s(2t-1)\},\; f \in \{0, \ell, \ell(t+1), \ell(t-1), \ell(2t-1)\} \,\Bigr\}.
\]  
 Then for each $r \in R$, Formula \eqref{polequiv} gives one equation in the unknowns $(\alpha,\beta,\gamma,\delta)$, subject to $\alpha\delta - \beta\gamma \neq 0$.  

To analyze possible coincidences among the exponents in $R$, we arrange them in  a $5 \times 5$ table, where rows correspond to $f$ and columns to $e$. This makes it clear that there cannot be two equal exponents in the same row or column, so any equality between exponents can only occur between entries in different rows and columns.  

\renewcommand{\arraystretch}{1.3}
\setlength{\tabcolsep}{8pt}

$$
\begin{array}{|c|c|c|c|c|}
\hline
0 & \ell & \ell(t+1) & \ell(t-1) & \ell(2t-1) \\ \hline
s & \ell+s & \ell(t+1)+s & \ell(t-1)+s & \ell(2t-1)+s \\ \hline
s(t+1) & \ell+s(t+1) & \ell(t+1)+s(t+1) & \ell(t-1)+s(t+1) & \ell(2t-1)+s(t+1) \\ \hline
s(t-1) & \ell+s(t-1) & \ell(t+1)+s(t-1) & \ell(t-1)+s(t-1) & \ell(2t-1)+s(t-1) \\ \hline
s(2t-1) & \ell+s(2t-1) & \ell(t+1)+s(2t-1) & \ell(t-1)+s(2t-1) & \ell(2t-1)+s(2t-1) \\ \hline
\end{array}
$$

In the following the table of the correspondent coefficients.

\begin{center}
\renewcommand{\arraystretch}{1.4}
\setlength{\tabcolsep}{10pt}
\Large
\resizebox{\textwidth}{!}{$
\begin{array}{|c|c|c|c|c|}
\hline
-\gamma & -\delta\mu & \delta\mu k^{1-q^{\ell(t+1)}} & -\delta & -\delta k^{1-q^{\ell(2t-1)}} \\ \hline
m\alpha^\qs & m\beta^\qs\mu^\qs & -m\beta^\qs\mu^\qs k^{\qs-q^{\ell(t+1)+s}} & m\beta^\qs & m\beta^\qs k^{\qs-q^{\ell(2t-1)+s}}  \\ \hline
-mh^{1-q^{s(t+1)}}\alpha^{q^{s(t+1)}} & -mh^{1-q^{s(t+1)}}\beta^{q^{s(t+1)}}\mu^{q^{s(t+1)}} & mh^{1-q^{s(t+1)}}\beta^{q^{s(t+1)}}\mu^{q^{s(t+1)}}k^{q^{s(t+1)}-q^{\ell(t+1)+s(t+1)}} & -mh^{1-q^{s(t+1)}}\beta^{q^{s(t+1)}} & -mh^{1-q^{s(t+1)}}\beta^{q^{s(t+1)}}k^{q^{s(t+1)}-q^{\ell(2t-1)+s(t+1)}} \\ \hline
\alpha^{q^{s(t-1)}} &\beta^{q^{s(t-1)}}\mu^{q^{s(t-1)}} & -\beta^{q^{s(t-1)}}\mu^{q^{s(t-1)}}k^{q^{s(t-1)}-q^{\ell(t+1)+s(t-1)}} & \beta^{q^{s(t-1)}} & \beta^{q^{s(t-1)}}k^{q^{s(t-1)}-q^{\ell(2t-1)+s(t-1)}} \\ \hline
h^{1-q^{s(2t-1)}}\alpha^{q^{s(2t-1)}} & h^{1-q^{s(2t-1)}}\beta^{q^{s(2t-1)}}\mu^{q^{s(2t-1)}} & -h^{1-q^{s(2t-1)}}\beta^{q^{s(2t-1)}}\mu^{q^{s(2t-1)}}k^{q^{s(2t-1)}-q^{\ell(t+1)+s(2t-1)}} & h^{1-q^{s(2t-1)}}\beta^{q^{s(2t-1)}} & h^{1-q^{s(2t-1)}}\beta^{q^{s(2t-1)}}k^{q^{s(2t-1)}-q^{\ell(2t-1)+s(2t-1)}} \\ \hline
\end{array}$}
\end{center}

By comparing the exponent $0$ with all the others, the possible equalities (modulo $2t)$ reduce to the following cases:
\begin{enumerate}
    \item $0$ is different from all other exponents;
    \item $  \ell \equiv -s \pmod{2t}$;
    \item $\ell \equiv s \pmod{2t}$;
    \item $\ell  \equiv t-s \pmod{2t}$;
    \item $ \ell \equiv t+s \pmod{2t}$. 
    
\end{enumerate}

\paragraph{Case (a):} \emph{$0$ is different from all other exponents}.

From the coefficient of $x$, we obtain $\gamma = 0.$ Now, comparing $s$ and $\ell$, respectively,  with the other exponents and recalling that $\gcd(s,2t)=1=\gcd(\ell,2t)$, we get  that only the following two types of nontrivial equalities can occur:

\begin{itemize}
  \item $2s - \ell \equiv t \pmod{2t} \quad\text{or}\quad 2s + \ell \equiv t \pmod{2t}.$
  \item $ 2\ell - s \equiv t \pmod{2t} \quad\text{or}\quad 2\ell + s \equiv t \pmod{2t}.$
\end{itemize}

Moreover, these four possibilities are pairwise incompatible: a solution of the first type cannot satisfy a solution of the second type when $t\geq5$.  
Hence, either $s$ is different from all other exponents or $\ell$ is different from all other exponents. Therefore the corresponding coefficient in the matrix must vanish.  In particular,

\begin{itemize}
  \item if $s$ is different from all other exponents, then $\alpha = 0$;
  \item if $\ell$ is different from all other exponents, then $\delta= 0$.
\end{itemize}

In any case, $M$ is not invertible, a contradiction.

\paragraph{Case (b): $\ell + s \equiv 0 \pmod{2t}$.}

Under the hypothesis $\ell\equiv -s\pmod{2t}$, the exponents
$e+f$ can be arranged in the following $5\times5$ table (all entries reduced to
residues in $\{0,\dots,2t-1\}$):

\[
\begin{array}{|c|c|c|c|c|}
\hline
0            & s(2t-1)        & s(t-1)     & s(t+1)       & s     \\ \hline
s     & 0           & s t          & s(t+2)      & 2s     \\ \hline
s(t+1)         &s t           & 0          & 2s     & s(t+2)      \\ \hline
s(t-1)      & s(t-2)      & s(2t-2)      & 0            &s t            \\ \hline
s(2t-1)           & s(2t-2)       & s(t-2)     &s t            & 0            \\ \hline
\end{array}
\]

\medskip

Any  other equality between two different exponents reduces to a congruence of the form
\[
2s\equiv 0,\qquad 4s \equiv t,\qquad 2s \equiv t,\qquad s\equiv t \pmod{2t}.
\]
Since $\gcd(\ell,2t)=1$ and $t\geq 5$, this is a contradiction.

The system given by the coefficients of $x^{q^{2s}},x^{q^{s(t-2)}},x^{q^{s(t+2)}},x^{q^{s(2t-2)}},x^{q^{s t}}$ is

$$\begin{cases}
    -\beta^{q^{s(t-1)}}\mu^{q^{s(t-1)}}k^{q^{s(t-1)}-q^{\ell(t+1)+s(t-1)}}+h^{1-q^{s(2t-1)}}\beta^{q^{s(2t-1)}}\mu^{q^{s(2t-1)}}=0\\
\beta^{q^{s(t-1)}}\mu^{q^{s(t-1)}}-h^{1-q^{s(2t-1)}}\beta^{q^{s(2t-1)}}\mu^{q^{s(2t-1)}}k^{q^{s(2t-1)}-q^{\ell(t+1)+s(2t-1)}}=0\\
m\beta^\qs-mh^{1-q^{s(t+1)}}\beta^{q^{s(t+1)}}k^{q^{s(t+1)}-q^{\ell(2t-1)+s(t+1)}}=0\\
m\beta^\qs k^{\qs-q^{\ell(2t-1)+s}}-mh^{1-q^{s(t+1)}}\beta^{q^{s(t+1)}}=0\\
-m\beta^\qs\mu^\qs k^{\qs-q^{\ell(t+1)+s}}-mh^{1-q^{s(t+1)}}\beta^{q^{s(t+1)}}\mu^{q^{s(t+1)}}+\beta^{q^{s(t-1)}}k^{q^{s(t-1)}-q^{\ell(2t-1)+s(t-1)}}+h^{1-q^{s(2t-1)}}\beta^{q^{s(2t-1)}}=0,
\end{cases}$$

which in this case is equivalent to
$$
\begin{cases}
     \beta^{q^{s(2t-1)}}=\beta^{q^{s(t-1)}}k^{q^{s(t-1)}-q^{s(2t-2)}}h^{q^{s(2t-1)}-1}\\
     \beta^{q^{s(2t-1)}}=\beta^{q^{s(t-1)}}h^{q^{s(2t-1)}-1}k^{q^{s(t-2)}-q^{s(2t-1)}}\\
    \beta^{q^{s(t+1)}} =\beta^\qs h^{q^{s(t+1)}-1}k^{q^{s(t+2)}-q^{s(t+1)}}\\
    \beta^{q^{s(t+1)}}=\beta^\qs h^{q^{s(t+1)}-1} k^{\qs-q^{2s}}\\
    -m\beta^\qs\mu^\qs k^{\qs-q^{st}}-mh^{1-q^{s(t+1)}}\beta^{q^{s(t+1)}}\mu^{\qs}+\beta^{q^{s(t-1)}}k^{q^{s(t-1)}-q^{st}}+h^{1-q^{s(2t-1)}}\beta^{q^{s(2t-1)}}=0.
\end{cases}
$$
The first and second equations are equivalent to 
\[
\beta^{q^{st}} = \beta\, k^{1-q^{s(t-1)}}\, h^{q^{st}-q^{s(t+1)}}.
\] 
Similarly, the third and fourth equations are equivalent to 
\[
\beta^{q^{st}} = \beta\, k^{1-q^s}\, h^{q^{st}-q^{s(2t-1)}}.
\] 
If $\beta \neq 0$, we the wen have
\[
k^{1-q^{s(t-1)}}\, h^{q^{st}-q^{s(t+1)}} = k^{1-q^s}\, h^{q^{st}-q^{s(2t-1)}},
\] 
which can be rewritten as 
\[
(kh)^{q^{s(t-2)}-1} = 1.
\] 
Setting $\lambda := kh$, we have  $\lambda \in \mathbb{F}_{q^{t-2}} \cap \mathbb{F}_{q^{2t}} = \mathbb{F}_{q^{\gcd(2t,t-2)}}$ and 
\[
N:=\mathrm{N}_{q^{2t}/q^t}(\lambda)= \mathrm{N}_{q^{2t}/q^t}(h) \, \mathrm{N}_{q^{2t}/q^t}(k) = \pm 1.
\]
If $t$ is odd, $\lambda\in\F_q$, 
so $\lambda^2=N$. If $t\equiv0 \pmod{4}$, $\lambda\in\F_{q^2}$, and $\lambda^2=N$ again. Instead, if $t\equiv2 \pmod{4}$ then $\lambda\in\F_{q^4}$ and $\lambda^{q^2+1}=N$.
Let us replace $$\beta^{q^{s(2t-1)}}=\beta^{q^{s(t-1)}}k^{q^{s(t-1)}-q^{s(2t-2)}}h^{q^{s(2t-1)}-1} \,\,\textnormal{ and }\,\, \beta^{q^{s(t+1)}} =\beta^\qs h^{q^{s(t+1)}-1}k^{q^{s(t+2)}-q^{s(t+1)}}$$  in the fifth equation. Then,

$$m\beta^\qs\mu^\qs( k^{\qs-q^{st}}+k^{q^{s(t+2)}-q^{s(t+1)}})=\beta^{q^{s(t-1)}}(k^{q^{s(t-1)}-q^{st}}+k^{q^{s(t-1)}-q^{s(2t-2)}}).$$

So

$$m\mu^\qs( k^{\qs-q^{st}}+k^{-q^{2s}+q^{s}})=\beta^{q^{s(t-1)}-\qs}(k^{q^{s(t-1)}-q^{st}}+k^{q^{s(t-1)}-q^{s(2t-2)}}),$$ 

and we have

$$m\mu^\qs=\big(\beta^\qs k^{\qs}( k^{-q^{st}}+k^{-q^{2s}})\big)^{q^{s(t-2)}-1}.$$

Therefore,
\begin{eqnarray*}\big(\beta^\qs k^{\qs}( k^{-q^{st}}+k^{-q^{2s}})\big)^{q^{st}}&=&\beta^{q^{s(t+1)}} k^{q^{s(t+1)}}( k^{-1}+k^{-q^{s(t+2)}})\\&=&\beta^\qs h^{q^{s(t+1)}-1}k^{q^{s(t+2)}-q^{s(t+1)}}k^{q^{s(t+1)}}( k^{-1}+k^{-q^{s(t+2)}})\\&=&\beta^\qs \lambda^{q^{s(t+1)}-1}k^{1-q^{s(t+1)}}k^{q^{s(t+2)}}( k^{-1}+k^{-q^{s(t+2)}})\\&=&\beta^\qs \lambda^{q^{s(t+1)}-1}k^{-q^{s(t+1)}+1}k^{q^{s(t+2)}}( k^{-1}+k^{-q^{s(t+2)}})\\&=&\beta^\qs \lambda^{q^{3s}-1}k^{-q^{s(t+1)}}( k+k^{q^{s(t+2)}})\\&=&\beta^\qs \lambda^{q^{3s}-1}k^{\qs}( k^{-q^{st}}+k^{-q^{2s}}).
\end{eqnarray*}

and, hence, $m\mu^\qs=z^{q^{s(t-2)}-1}$ with $z^{q^{st}}=\lambda^{q^{3s}-1}z$.

If $\beta=0$, then $\gamma=0$. Hence, $\alpha\delta\neq0$ and consider the system given by the coefficients of $x^\qs,x^{q^{s(t+1)}},x^{q^{s(t-1)}},x^{q^{s(2t-1)}}$

$$\begin{cases}
    -\delta\mu+h^{1-q^{s(2t-1)}}\alpha^{q^{s(2t-1)}}=0\\
\delta\mu k^{1-q^{\ell(t+1)}}+\alpha^{q^{s(t-1)}}=0\\
 \delta +mh^{1-q^{s(t+1)}}\alpha^{q^{s(t+1)}}=0\\
 -\delta k^{1-q^{\ell(2t-1)}} +m\alpha^\qs=0,
\end{cases}$$
 which in this case is equivalent to

 $$\begin{cases}
    -\delta\mu+h^{1-q^{s(2t-1)}}\alpha^{q^{s(2t-1)}}=0\\
\delta\mu k^{1-q^{s(t-1)}}+\alpha^{q^{s(t-1)}}=0\\
 \delta= -mh^{1-q^{s(t+1)}}\alpha^{q^{s(t+1)}}\\
 -\delta k^{1-q^{s}} +m\alpha^\qs=0.
\end{cases}$$
Let us substitute $\delta$ from the third equation into the others, then we obtain
 $$\begin{cases}
    mh^{1-q^{s(t+1)}}\alpha^{q^{s(t+1)}}\mu+h^{1-q^{s(2t-1)}}\alpha^{q^{s(2t-1)}}=0\\
-mh^{1-q^{s(t+1)}}\alpha^{q^{s(t+1)}}\mu k^{1-q^{s(t-1)}}+\alpha^{q^{s(t-1)}}=0\\
 \delta= -mh^{1-q^{s(t+1)}}\alpha^{q^{s(t+1)}}\\
 h^{1-q^{s(t+1)}}\alpha^{q^{s(t+1)}} k^{1-q^{s}} +\alpha^\qs=0.
\end{cases}$$
From the first two equations, $\alpha^{q^{s(2t-1)}}=-\alpha^{q^{s(t-1)}}k^{q^{s(t-1)}-1}h^{q^{s(2t-1)}-1}$ and hence $\alpha^{q^{st}}=-\alpha k^{1-q^{s(t+1)}}h^{q^{st}-q^{s(t+1)}}$.\\
From the fourth equation,$\alpha^{q^{s(t+1)}}=-\alpha^\qs h^{q^{s(t+1)}-1} k^{q^{s}-1}$, that is $\alpha^{q^{st}}=-\alpha h^{q^{st}-q^{s(2t-1)}} k^{1-q^{s(2t-1)}}$. Then,
$$k^{1-q^{s(t+1)}}h^{q^{st}-q^{s(t+1)}}=h^{q^{st}-q^{s(2t-1)}} k^{1-q^{s(2t-1)}}$$
that reads again as $(kh)^{q^{s(t-2)}-1}=1$. We have that $\lambda:=kh$ has the same properties as before.

Now from the second equation we obtain 

\begin{eqnarray*}m\mu&=&\alpha^{q^{s(t-1)}}h^{q^{s(t+1)}-1} k^{q^{s(t-1)}-1}\alpha^{-q^{s(t+1)}}\\&=&-\alpha^{q^{s(t-1)}-\qs}h^{q^{s(t+1)}-1} k^{q^{s(t-1)}-1}h^{1-q^{s(t+1)}} k^{1-q^{s}}\\&=&-\alpha^{q^{s(t-1)}-\qs}h^{q^{s(t+1)}-1} k^{q^{s(t-1)}-1}h^{1-q^{s(t+1)}} k^{1-q^{s}}\\&=&-\alpha^{q^{s(t-1)}-\qs} k^{q^{s(t-1)}-\qs}\\ &=&-\big(\alpha^\qs k^\qs)^{q^{s(t-2)}-1}.
\end{eqnarray*}
Moreover,
\begin{eqnarray*}
  (\alpha^\qs k^\qs)^{q^{st}}&=& -\alpha^\qs h^{q^{s(t+1)}-1} k^{q^{s}-1}k^{q^{s(t+1)}}\\&=&
   -\alpha^\qs \lambda^{q^{s(t+1)}-1}k^{-q^{s(t+1)}+1} k^{q^{s}-1}k^{q^{s(t+1)}}\\&=&-\alpha^\qs\lambda^{q^{3s}-1}k^\qs.
\end{eqnarray*}

So $m\mu=-z^{q^{s(t-2)}-1}$ with $z^{q^{st}}=-\lambda^{q^{3s}-1}z$.

\paragraph{Case (c): $\ell \equiv s  \pmod{2t}$.}

Under the hypothesis $\ell \equiv s \pmod{2t}$ (and $\gcd(s,2t)=\gcd(\ell,2t)=1$), the exponents
$e+f $ modulo $2t$  can be arranged in the following $5\times5$ table:

\[
\begin{array}{|c|c|c|c|c|}
\hline
0            & s        & s(t+1)     & s(t-1)       & s(2t-1)      \\ \hline
s      & 2s           & s(t+2)          & st      & 0     \\ \hline
s(t+1)       &s(t+2)           & 2s          & 0     & st      \\ \hline
s(t-1)      & st    & 0      & s(2t-2)            &s(t-2)            \\ \hline
s(2t-1)         & 0       & st    &s(t-2)            & s(2t-2)            \\ \hline
\end{array}
\]

The system given by the coefficients of $x^{q^{2s}},x^{q^{s(t+2)}},x^{q^{s(t-2)}},x^{q^{s(2t-2)}},x^{q^{s t}}$ is

$$\begin{cases}
     m\beta^\qs\mu^\qs+mh^{1-q^{s(t+1)}}\beta^{q^{s(t+1)}}\mu^{q^{s(t+1)}}k^{q^{s(t+1)}-q^{\ell(t+1)+s(t+1)}}=0\\
 -m\beta^\qs\mu^\qs k^{\qs-q^{\ell(t+1)+s}}-mh^{1-q^{s(t+1)}}\beta^{q^{s(t+1)}}\mu^{q^{s(t+1)}}=0\\
 \beta^{q^{s(t-1)}}+ h^{1-q^{s(2t-1)}}\beta^{q^{s(2t-1)}}k^{q^{s(2t-1)}-q^{\ell(2t-1)+s(2t-1)}}=0\\
 \beta^{q^{s(t-1)}}k^{q^{s(t-1)}-q^{\ell(2t-1)+s(t-1)}}+ h^{1-q^{s(2t-1)}}\beta^{q^{s(2t-1)}}=0\\
 m\beta^\qs-mh^{1-q^{s(t+1)}}\beta^{q^{s(t+1)}}k^{q^{s(t+1)}-q^{\ell(2t-1)+s(t+1)}}+\\+\beta^{q^{s(t-1)}}\mu^{q^{s(t-1)}}-h^{1-q^{s(2t-1)}}\beta^{q^{s(2t-1)}}\mu^{q^{s(2t-1)}}k^{q^{s(2t-1)}-q^{\ell(t+1)+s(2t-1)}}=0,
\end{cases}$$

which in this case is equivalent to
$$\begin{cases} \label{thisystem}
   \beta^{q^{s(t+1)}}= 
   -\beta^\qs h^{q^{s(t+1)}-1}k^{q^{2s}-q^{s(t+1)}}\\
  \beta^{q^{s(t+1)}}=-\beta^\qs k^{\qs-q^{s(t+2)}} h^{q^{s(t+1)}-1}\\
  \beta^{q^{s(2t-1)}}=-\beta^{q^{s(t-1)}}h^{q^{s(2t-1)}-1}k^{q^{s(2t-2)}-q^{s(2t-1)}}\\
  \beta^{q^{s(2t-1)}}=-\beta^{q^{s(t-1)}}k^{q^{s(t-1)}-q^{s(t-2)}} h^{q^{s(2t-1)}-1}\\
  m\beta^\qs-mh^{1-q^{s(t+1)}}\beta^{q^{s(t+1)}}k^{q^{s(t+1)}-q^{st}}+\beta^{q^{s(t-1)}}\mu^{q^{s(t-1)}}-h^{1-q^{s(2t-1)}}\beta^{q^{s(2t-1)}}\mu^{q^{s(t-1)}}k^{q^{s(2t-1)}-q^{st}}=0.
\end{cases}$$

The first and second equations are equivalent to 
\[
\beta^{q^{st}} = -\beta\, h^{q^{st}-q^{s(2t-1)}}\, k^{q^s - q^{st}}.
\] 
Similarly, the third and fourth equations are equivalent to 
\[
\beta^{q^{st}} = -\beta\, h^{q^{st}-q^{s(t+1)}}\, k^{q^{s(t-1)} - q^{st}}.
\] 
If $\beta \neq 0$, we then obtain
\[
h^{q^{st}-q^{s(2t-1)}}\, k^{q^s - q^{st}} = h^{q^{st}-q^{s(t+1)}}\, k^{q^{s(t-1)} - q^{st}},
\] 
which can be rewritten as 
\[
\left(\frac{h}{k}\right)^{q^{s(t-2)} - 1} = 1.
\] 
Setting $\nu := h/k$, we see that $\nu$ has the same properties as $\lambda$ described as in \textbf{Case (b)}.

Let us replace 
$$\beta^{q^{s(2t-1)}}=-\beta^{q^{s(t-1)}}h^{q^{s(2t-1)}-1}k^{q^{s(2t-2)}-q^{s(2t-1)}} \,\, \textnormal{ and } \,\, \beta^{q^{s(t+1)}}= 
   -\beta^\qs h^{q^{s(t+1)}-1}k^{q^{2s}-q^{s(t+1)}}$$  in the fifth equation of System \eqref{thisystem}.

The,

$$mk^{q^{2s}}(k^{-q^{2s}}+k^{-q^{st}})=-\beta^{q^{s(t-1)}-\qs}\mu^{q^{s(t-1)}}k^{q^{s(2t-2)}}(k^{-q^{s(2t-2)}}+k^{-q^{st}}),$$ that is equivalent to
$$m\mu^{-q^{s(t-1)}}=-\big(\beta^\qs k^{q^{st}+q^{2s}}(k^{-q^{2s}}+k^{-q^{st}})\big)^{q^{s(t-2)}-1}=-\big(\beta^\qs(k^{q^{2s}}+k^{q^{st}})\big)^{q^{s(t-2)}-1}.$$

Hence, we have \begin{eqnarray*}
   \big(\beta^\qs(k^{q^{2s}}+k^{q^{st}})\big)^{q^{st}}&=&\beta^{q^{s(t+1)}}(k^{q^{s(t+2)}}+k)\\&=&-\beta^\qs h^{q^{s(t+1)}-1}k^{q^{2s}-q^{s(t+1)}}(k^{q^{s(t+2)}}+k)\\&=&-\beta^\qs\nu^{q^{s(t+1)}-1}k^{q^{s(t+1)}-1}k^{q^{2s}-q^{s(t+1)}}(k^{q^{s(t+2)}}+k)\\&=&-\beta^\qs\nu^{q^{3s}-1}k^{q^{2s}-1}(k^{q^{s(t+2)}}+k)\\&=&-\beta^\qs\nu^{q^{3s}-1}(k^{q^{2s}}+k^{q^{st}}).
\end{eqnarray*}
This implies that $m\mu^{-q^{s(t-1)}}=-z^{q^{s(t-2)}-1}$ with $z^{q^{st}}=-\nu^{q^{3s}-1}z$.

If $\beta=0$, then $\gamma=0$. Hence, $\alpha\delta\neq0$ and consider the system given by the coefficients of $x^\qs,x^{q^{s(t+1)}},x^{q^{s(t-1)}},x^{q^{s(2t-1)}}$
$$\begin{cases}
    -\delta\mu+m\alpha^\qs=0\\
\delta\mu k^{1-q^{\ell(t+1)}}-mh^{1-q^{s(t+1)}}\alpha^{q^{s(t+1)}}=0\\
 -\delta+\alpha^{q^{s(t-1)}}=0\\
 -\delta k^{1-q^{\ell(2t-1)}}+h^{1-q^{s(2t-1)}}\alpha^{q^{s(2t-1)}} =0,
\end{cases}$$
which is equivalent to

 $$\begin{cases}
   -\delta\mu+m\alpha^\qs=0\\
\delta\mu k^{1-q^{s(t+1)}}-mh^{1-q^{s(t+1)}}\alpha^{q^{s(t+1)}}=0\\
\delta=\alpha^{q^{s(t-1)}}\\
-\delta k^{1-q^{s(2t-1)}}+h^{1-q^{s(2t-1)}}\alpha^{q^{s(2t-1)}} =0.
\end{cases}$$
Let us substitute $\delta$ from the third equation into the others, then we obtain
 $$\begin{cases} \label{systemcasec}
     -\alpha^{q^{s(t-1)}}\mu+m\alpha^\qs=0\\
\alpha^{q^{s(t-1)}}\mu k^{1-q^{s(t+1)}}-mh^{1-q^{s(t+1)}}\alpha^{q^{s(t+1)}}=0\\
\delta=\alpha^{q^{s(t-1)}}\\
-\alpha^{q^{s(t-1)}} k^{1-q^{s(2t-1)}}+h^{1-q^{s(2t-1)}}\alpha^{q^{s(2t-1)}} =0.
\end{cases}$$

From the first two equations, we obtain
\[
\alpha^{q^{s(t+1)}} = \alpha^\qs\, k^{1-q^{s(t+1)}}\, h^{q^{s(t+1)}-1},
\]
which can be rewritten as
\[
\alpha^{q^{st}} = \alpha\, k^{q^{s(2t-1)}-q^{st}}\, h^{q^{st}-q^{s(2t-1)}}.
\]
Similarly, from the fourth equation, we obtain
\[
\alpha^{q^{s(2t-1)}} = \alpha^{q^{s(t-1)}}\, h^{q^{s(2t-1)}-1}\, k^{1-q^{s(2t-1)}},
\]
and hence
\[
\alpha^{q^{st}} = \alpha\, h^{q^{st}-q^{s(t+1)}}\, k^{q^{s(t+1)}-q^{st}}.
\]
Comparing these expressions, we then obtain
\[
k^{q^{s(2t-1)}-q^{st}}\, h^{q^{st}-q^{s(2t-1)}} = h^{q^{st}-q^{s(t+1)}}\, k^{q^{s(t+1)}-q^{st}},
\]
which can be rewritten as 
\[
\left(\frac{h}{k}\right)^{q^{s(t-2)}-1} = 1.
\]
Hence, setting $\nu := h/k$, we see that $\nu$ has the same properties as before.

Now from the second equation of System \eqref{systemcasec}, we obtain

\begin{eqnarray*}
    m/\mu&=&\alpha^{q^{s(t-1)}} k^{1-q^{s(t+1)}}h^{q^{s(t+1)}-1}\alpha^{-q^{s(t+1)}}\\&=&\alpha^{q^{s(t-1)}-\qs} k^{1-q^{s(t+1)}}h^{q^{s(t+1)}-1} k^{q^{s(t+1)-1}}h^{1-q^{s(t+1)}}\\&=&(\alpha^\qs)^{q^{s(t-2)}-1}.
\end{eqnarray*}
Finally, we have that $(\alpha^{q^s})^{q^{st}}=\alpha^\qs k^{1-q^{s(t+1)}}h^{q^{s(t+1)}-1}=\alpha^{q^s}\nu^{q^{3s}-1}$, so $m/\mu=z^{q^{s(t-2)}-1}$ with $z^{q^{st}}=\nu^{q^{3s}-1}z$.

\paragraph{Case (d): $\ell \equiv t-s \pmod{2t} $.}

Note that this case is possible only when $t$ is even.
Under the hypothesis $\ell \equiv t-s \pmod{2t} $, $t^2 \equiv 0 \pmod{2t}$ and the exponents
$e+f$ (modulo $2t$) can be arranged in the following $5\times5$ table:
\[
\begin{array}{|c|c|c|c|c|}
\hline
0            & s(t-1)        & s(2t-1)     & s       & s(t+1)     \\ \hline
s      & st           & 0          & 2s      & s(t+2)     \\ \hline
s(t+1)       &0          & st          & s(t+2)     & 2s      \\ \hline
s(t-1)      & s(2t-2)    & 2(t-2)      & st            &0            \\ \hline
s(2t-1)         & s(t-2)       & s(2t-2)    &0            & st            \\ \hline
\end{array}
\]

The system given by the coefficients of $x^{q^{2s}},x^{q^{s(t-2)}},x^{q^{s(t+2)}},x^{q^{s(2t-2)}},x^{q^{s t}}$ is

$$\begin{cases}
    -\beta^{q^{s(t-1)}}\mu^{q^{s(t-1)}}k^{q^{s(t-1)}-q^{\ell(t+1)+s(t-1)}}+h^{1-q^{s(2t-1)}}\beta^{q^{s(2t-1)}}\mu^{q^{s(2t-1)}}=0\\
\beta^{q^{s(t-1)}}\mu^{q^{s(t-1)}}-h^{1-q^{s(2t-1)}}\beta^{q^{s(2t-1)}}\mu^{q^{s(2t-1)}}k^{q^{s(2t-1)}-q^{\ell(t+1)+s(2t-1)}}=0\\
m\beta^\qs-mh^{1-q^{s(t+1)}}\beta^{q^{s(t+1)}}k^{q^{s(t+1)}-q^{\ell(2t-1)+s(t+1)}}=0\\
m\beta^\qs k^{\qs-q^{\ell(2t-1)+s}}-mh^{1-q^{s(t+1)}}\beta^{q^{s(t+1)}}=0\\
 m\beta^\qs\mu^\qs+mh^{1-q^{s(t+1)}}\beta^{q^{s(t+1)}}\mu^{q^{s(t+1)}}k^{q^{s(t+1)}-q^{\ell(t+1)+s(t+1)}}+\\+\beta^{q^{s(t-1)}}+h^{1-q^{s(2t-1)}}\beta^{q^{s(2t-1)}}k^{q^{s(2t-1)}-q^{\ell(2t-1)+s(2t-1)}}=0,
\end{cases}$$

which in this case is equivalent to
$$
\begin{cases}\label{case(d)}
     \beta^{q^{s(2t-1)}}=\beta^{q^{s(t-1)}}k^{q^{s(t-1)}-q^{s(t-2)}}h^{q^{s(2t-1)}-1}\\
     \beta^{q^{s(2t-1)}}=\beta^{q^{s(t-1)}}h^{q^{s(2t-1)}-1}k^{q^{s(2t-2)}-q^{s(2t-1)}}\\
    \beta^{q^{s(t+1)}} =\beta^\qs h^{q^{s(t+1)}-1}k^{q^{2s}-q^{s(t+1)}}\\
    \beta^{q^{s(t+1)}}=\beta^\qs h^{q^{s(t+1)}-1} k^{\qs-q^{s(t+2)}}\\
    m\beta^\qs\mu^\qs+mh^{1-q^{s(t+1)}}\beta^{q^{s(t+1)}}\mu^{q^{s}}k^{q^{s(t+1)}-q^{st}}+\beta^{q^{s(t-1)}}+h^{1-q^{s(2t-1)}}\beta^{q^{s(2t-1)}}k^{q^{s(2t-1)}-q^{st}}=0.
\end{cases}
$$
The first and second equation of the system above are equivalent to $\beta^{q^{st}}=\beta k^{1-q^{s(2t-1)}}h^{q^{st}-q^{s(t+1)}}$. On the other hand, the third and fourth equations are equivalent to $\beta^{q^{st}}=\beta k^{1-q^{s(t+1)}}h^{q^{st}-q^{s(2t-1)}}$. If $\beta\neq0$, we have $$k^{1-q^{s(2t-1)}}h^{q^{st}-q^{s(t+1)}}= k^{1-q^{s(t+1)}}h^{q^{st}-q^{s(2t-1)}},$$ that reads as $(h/k)^{q^{s(t-2)}-1}=1$. So, $\nu:=h/k$ has the same properties as before.

Let us replace $$\beta^{q^{s(2t-1)}}=\beta^{q^{s(t-1)}}k^{q^{s(t-1)}-q^{s(t-2)}}h^{q^{s(2t-1)}-1} \,\, \textnormal{ and } \,\, \beta^{q^{s(t+1)}} =\beta^\qs h^{q^{s(t+1)}-1}k^{q^{2s}-q^{s(t+1)}}$$ in the fifth equation of System \eqref{case(d)}. Then,

$$m\mu^\qs\beta^\qs(k^{-1}+k^{-q^{s(t+2)}})=-\beta^{q^{s(t-1)}}(k^{-1}+k^{-q^{s(t-2)}}),$$ that it is equivalent to

$$m\mu^\qs=-\big(\beta^\qs (k^{-1}+k^{-q^{s(t+2)}})\big)^{q^{s(t-2)}-1}.$$

Hence, we have 
\begin{eqnarray*}\big(\beta^\qs (k^{-1}+k^{-q^{s(t+2)}})\big)^{q^{st}}&=&\beta^{q^{s(t+1)}}(k^{-q^{st}}+k^{-q^{2s}})\\&=&\beta^\qs h^{q^{s(t+1)}-1}k^{q^{2s}-q^{s(t+1)}}(k^{-q^{st}}+k^{-q^{2s}})\\&=&\beta^\qs\nu^{q^{s(t+1)}-1} k^{q^{s(t+1)}-1}k^{q^{2s}-q^{s(t+1)}}(k^{-q^{st}}+k^{-q^{2s}})\\&=&\beta^\qs\nu^{q^{s(t+1)}-1}k^{q^{2s}-1}(k^{-q^{st}}+k^{-q^{2s}})\\&=&\beta^\qs\nu^{q^{s(t+1)}-1}k^{-q^{s(t+2)}-1}(k+k^{q^{s(t+2)}})\\&=&\beta^\qs\nu^{q^{3s}-1} (k^{-1}+k^{-q^{s(t+2)}}).
\end{eqnarray*}

We get that $m\mu^\qs=-z^{q^{s(t-2)}-1}$ with $z^{q^{st}}=\nu^{q^{3s}-1}z$.

If $\beta=0$, then $\gamma=0$. Hence, $\alpha\delta\neq0$ and consider the system given by the coefficients of $x^\qs,x^{q^{s(t+1)}},x^{q^{s(t-1)}},x^{q^{s(2t-1)}}$

$$\begin{cases}
    -\delta\mu+\alpha^{q^{s(t-1)}}=0\\
\delta\mu k^{1-q^{\ell(t+1)}}+h^{1-q^{s(2t-1)}}\alpha^{q^{s(2t-1)}}=0\\
 -\delta+m\alpha^\qs =0\\
 -\delta k^{1-q^{\ell(2t-1)}} -mh^{1-q^{s(t+1)}}\alpha^{q^{s(t+1)}}=0,
\end{cases}$$
 which in this case is equivalent to

 $$\begin{cases}
    -\delta\mu+\alpha^{q^{s(t-1)}}=0\\
\delta\mu k^{1-q^{s(2t-1)}}+h^{1-q^{s(2t-1)}}\alpha^{q^{s(2t-1)}}=0\\
 \delta=m\alpha^\qs \\
 -\delta k^{1-q^{s(t+1)}} -mh^{1-q^{s(t+1)}}\alpha^{q^{s(t+1)}}=0.
\end{cases}$$
Let us substitute $\delta$ from the third equation into the others of the system above. Then, we obtain
$$\begin{cases}\label{case(d)1}
    -m\alpha^\qs\mu+\alpha^{q^{s(t-1)}}=0\\
m\alpha^\qs\mu k^{1-q^{s(2t-1)}}+h^{1-q^{s(2t-1)}}\alpha^{q^{s(2t-1)}}=0\\
 \delta=m\alpha^\qs \\
 \alpha^\qs k^{1-q^{s(t+1)}}+ h^{1-q^{s(t+1)}}\alpha^{q^{s(t+1)}}=0,
\end{cases}$$
From the first two equations of the system above,
\[
\alpha^{q^{s(2t-1)}} = -\alpha^{q^{s(t-1)}}\, k^{1-q^{s(2t-1)}}\, h^{q^{s(2t-1)}-1},
\]
which can be rewritten as
\[
\alpha^{q^{st}} = -\alpha\, k^{q^{s(t+1)}-q^{st}}\, h^{q^{st}-q^{s(t+1)}}.
\]
Similarly, from the fourth equation, 
\[
\alpha^{q^{s(t+1)}} = -\alpha^\qs\, h^{q^{s(t+1)}-1}\, k^{1-q^{s(t+1)}},
\]
which can be rewritten as
\[
\alpha^{q^{st}} = -\alpha\, h^{q^{st}-q^{s(2t-1)}}\, k^{q^{s(2t-1)}-q^{st}}.
\]
Comparing these two expressions, then
\[
k^{q^{s(t+1)}-q^{st}}\, h^{q^{st}-q^{s(t+1)}} = h^{q^{st}-q^{s(2t-1)}}\, k^{q^{s(2t-1)}-q^{st}},
\]
which can be rewritten as
\[
\left(\frac{h}{k}\right)^{q^{s(t-2)}-1} = 1.
\]
Setting $\nu := h/k$, we see that $\nu$ satisfies the same properties as before.

Now from the second equation of the System \eqref{case(d)1}, we have
\begin{eqnarray*}m\mu&=&-h^{1-q^{s(2t-1)}}\alpha^{q^{s(2t-1)}} k^{q^{s(2t-1)}-1}\alpha^{-\qs}\\&=&h^{1-q^{s(2t-1)}}\alpha^{q^{s(t-1)}}k^{1-q^{s(2t-1)}}h^{q^{s(2t-1)}-1} k^{q^{s(2t-1)}-1}\alpha^{-\qs}\\&=&\big(\alpha^\qs\big)^{q^{s(t-2)}-1},
\end{eqnarray*}
and, hence,
\begin{eqnarray*}
  (\alpha^\qs )^{q^{st}}= -\alpha^\qs h^{q^{s(t+1)}-1} k^{1-q^{s(t+1)}}=-\alpha^\qs\nu^{q^{3s}-1}.
\end{eqnarray*}

So, $m\mu=z^{q^{s(t-2)}-1}$ with $z^{q^{st}}=-\nu^{q^{3s}-1}z$.

\paragraph{Case (e): $\ell \equiv t+s \pmod{2t}$.}

Note that this case is possible only when $t$ is even.
Under the hypothesis $\ell \equiv t+s \pmod{2t} $,  $t^2\equiv0\pmod{2t}$ and  the exponents
$e+f \pmod{2t}$ can arranged in the following $5\times5$ table:

\[
\begin{array}{|c|c|c|c|c|}
\hline
0                   & s(t+1) &s          & s(2t-1)   & s(t-1)    \\ \hline
s                & s(t+2)    & 2s             & 0  & st   \\ \hline
s(t+1)                & 2s      &s(t+2)          & st  & 0     \\ \hline
s(t-1)        & 0   & st          &s(t-2)     & s(2t-2)              \\ \hline
s(2t-1)             & st  & 0      & s(2t-2) & s(t-2)               \\ \hline
\end{array}
\]

The system given by the coefficients of $x^{q^{2s}},x^{q^{s(t-2)}},x^{q^{s(t+2)}},x^{q^{s(2t-2)}},x^{q^{s t}}$ is

$$\begin{cases}
     m\beta^\qs\mu^\qs+mh^{1-q^{s(t+1)}}\beta^{q^{s(t+1)}}\mu^{q^{s(t+1)}}k^{q^{s(t+1)}-q^{\ell(t+1)+s(t+1)}}=0\\
 -m\beta^\qs\mu^\qs k^{\qs-q^{\ell(t+1)+s}}-mh^{1-q^{s(t+1)}}\beta^{q^{s(t+1)}}\mu^{q^{s(t+1)}}=0\\
 \beta^{q^{s(t-1)}}+ h^{1-q^{s(2t-1)}}\beta^{q^{s(2t-1)}}k^{q^{s(2t-1)}-q^{\ell(2t-1)+s(2t-1)}}=0\\
 \beta^{q^{s(t-1)}}k^{q^{s(t-1)}-q^{\ell(2t-1)+s(t-1)}}+ h^{1-q^{s(2t-1)}}\beta^{q^{s(2t-1)}}=0\\
m\beta^\qs k^{\qs-q^{\ell(2t-1)+s}} -mh^{1-q^{s(t+1)}}\beta^{q^{s(t+1)}}+\\-\beta^{q^{s(t-1)}}\mu^{q^{s(t-1)}}k^{q^{s(t-1)}-q^{\ell(t+1)+s(t-1)}}
 +h^{1-q^{s(2t-1)}}\beta^{q^{s(2t-1)}}\mu^{q^{s(2t-1)}}=0,
\end{cases}$$

which in this case is equivalent to
$$\begin{cases}\label{casee}
   \beta^{q^{s(t+1)}}= 
   -\beta^\qs h^{q^{s(t+1)}-1}k^{q^{s(t+2)}-q^{s(t+1)}}\\
  \beta^{q^{s(t+1)}}=-\beta^\qs k^{\qs-q^{2s}} h^{q^{s(t+1)}-1}\\
  \beta^{q^{s(2t-1)}}=-\beta^{q^{s(t-1)}}h^{q^{s(2t-1)}-1}k^{q^{s(t-2)}-q^{s(2t-1)}}\\
  \beta^{q^{s(2t-2)}}=-\beta^{q^{s(t-1)}}k^{q^{s(t-1)}-q^{s(2t-2)}} h^{q^{s(2t-1)}-1}\\
 m\beta^\qs k^{\qs-q^{st}} -mh^{1-q^{s(t+1)}}\beta^{q^{s(t+1)}}+\\-\beta^{q^{s(t-1)}}\mu^{q^{s(t-1)}}k^{q^{s(t-1)}-q^{st}}
 +h^{1-q^{s(2t-1)}}\beta^{q^{s(2t-1)}}\mu^{q^{s(t-1)}}=0.
\end{cases}$$

The first and second equations of the system above are equivalent to $\beta^{q^{st}}=-\beta h^{q^{st}-q^{s(2t-1)}}k^{q^{s(t+1)}-q^{st}}$. On the other hand, the third and fourth equations are equivalent to $\beta^{q^{st}}=-\beta h^{q^{st}-q^{s(t+1)}}k^{q^{s(2t-1)}-q^{st}}$.\\
So, if $\beta\neq0$, we obtain 
$$h^{q^{st}-q^{s(2t-1)}}k^{q^{s(t+1)}-q^{st}}=h^{q^{st}-q^{s(t+1)}}k^{q^{s(2t-1)}-q^{st}}$$ 
that is equivalent to $(hk)^{q^{s(t-2)}-1}=1$.
Then, $\lambda=hk$ that has the same properties as before.

Let us replace $$\beta^{q^{s(2t-1)}}=-\beta^{q^{s(t-1)}}h^{q^{s(2t-1)}-1}k^{q^{s(t-2)}-q^{s(2t-1)}} \,\, \textnormal{ and }\,\, \beta^{q^{s(t+1)}}= 
    -\beta^\qs h^{q^{s(t+1)}-1}k^{q^{s(t+2)}-q^{s(t+1)}}$$  in the fifth equation of System \eqref{casee}

$$m\beta^\qs( k^{\qs-q^{st}}+k^{q^{s(t+2)}-q^{s(t+1)}})-\beta^{q^{s(t-1)}}\mu^{q^{s(t-1)}}(k^{q^{s(t-1)}-q^{st}}
 +k^{q^{s(t-2)}-q^{s(2t-1)}})=0.$$

Then,

$$m\beta^\qs k^\qs( k+k^{q^{s(t+2)}})-\beta^{q^{s(t-1)}}\mu^{q^{s(t-1)}}k^{q^{s(t-1)}}(k
 +k^{q^{s(t-2)}})=0,$$ 
and hence,
$$m\mu^{-q^{s(t-1)}}=\beta^{q^{s(t-1)}-\qs}k^{q^{s(t-1)}-\qs}(k
 +k^{q^{s(t-2)}})/( k+k^{q^{s(t+2)}})=\big(\beta^\qs k^\qs (k+k^{q^{s(t+2)}})\big)^{q^{s(t-2)}-1}.$$

Again, \begin{eqnarray*}
   \big(\beta^\qs k^\qs (k+k^{q^{s(t+2)}})\big)^{q^{st}}&=&\beta^{q^{s(t+1)}}k^{q^{s(t+1)}}(k^{q^{st}}+k^{q^{2s}})\\&=&-\beta^\qs h^{q^{s(t+1)}-1}k^{q^{s(t+2)}-q^{s(t+1)}}k^{q^{s(t+1)}}(k^{q^{st}}+k^{q^{2s}})\\&=&-\beta^\qs k^{1-q^{s(t+1)}}\lambda^{q^{3s}-1}k^{q^{s(t+2)}}(k^{q^{st}}+k^{q^{2s}})\\&=&-\beta^\qs k^{1+\qs}\lambda^{q^{3s}-1}k^{q^{s(t+2)}}(k^{-1}+k^{-q^{s(t+2)}})\\&=&-\beta^\qs k^{\qs}\lambda^{q^{3s}-1}(k+k^{q^{s(t+2)}}),
\end{eqnarray*}
and $m\mu^{-q^{s(t-1)}}=z^{q^{s(t-2)}-1}$ with $z^{q^{st}}=-\lambda^{q^{3s}-1}z$.

If $\beta=0$, then $\gamma=0$. Hence, $\alpha\delta\ne 0$ and consider the system given by the coefficients of $x^\qs,x^{q^{s(t+1)}},x^{q^{s(t-1)}},x^{q^{s(2t-1)}}$
$$\begin{cases}
    -\delta\mu-mh^{1-q^{s(t+1)}}\alpha^{q^{s(t+1)}}=0\\
\delta\mu k^{1-q^{\ell(t+1)}}+m\alpha^\qs=0\\
 -\delta+h^{1-q^{s(2t-1)}}\alpha^{q^{s(2t-1)}}=0\\
 -\delta k^{1-q^{\ell(2t-1)}}+\alpha^{q^{s(t-1)}} =0,
\end{cases}$$
which in this case is equivalent to

 $$\begin{cases}\label{casee1}
    -\delta\mu-mh^{1-q^{s(t+1)}}\alpha^{q^{s(t+1)}}=0\\
\delta\mu k^{1-q^{s}}+m\alpha^\qs=0\\
 \delta=h^{1-q^{s(2t-1)}}\alpha^{q^{s(2t-1)}}\\
 -\delta k^{1-q^{s(t-1)}}+\alpha^{q^{s(t-1)}} =0.
\end{cases}$$
Let us substitute $\delta$ from the third equation into the others of the system above. Then, we obtain
 $$\begin{cases}
    -h^{1-q^{s(2t-1)}}\alpha^{q^{s(2t-1)}}\mu-mh^{1-q^{s(t+1)}}\alpha^{q^{s(t+1)}}=0\\
h^{1-q^{s(2t-1)}}\alpha^{q^{s(2t-1)}}\mu k^{1-q^{s}}+m\alpha^\qs=0\\
 \delta=h^{1-q^{s(2t-1)}}\alpha^{q^{s(2t-1)}}\\
 -h^{1-q^{s(2t-1)}}\alpha^{q^{s(2t-1)}} k^{1-q^{s(t-1)}}+\alpha^{q^{s(t-1)}} =0.
\end{cases}$$

From the first two equations of the system aboec, we obtain $\alpha^{q^{s(t+1)}}=\alpha^\qs k^{\qs-1}h^{q^{s(t+1)}-1}$
 that is equivalent to $$\alpha^{q^{st}}=\alpha k^{1-q^{s(2t-1)}}h^{q^{st}-q^{s(2t-1)}}.$$ On the other hand, from the fourth equation, we obtain $\alpha^{q^{s(2t-1)}}=\alpha^{q^{s(t-1)}} h^{q^{s(2t-1)}-1} k^{q^{s(t-1)}-1}$, and hence  $$\alpha^{q^{st}}=\alpha h^{q^{st}-q^{s(t+1)}} k^{1-q^{s(t+1)}}.$$ Then, $$k^{1-q^{s(2t-1)}}h^{q^{st}-q^{s(2t-1)}}=h^{q^{st}-q^{s(t+1)}} k^{1-q^{s(t+1)}},$$ that reads again as $(hk)^{q^{s(t-2)}-1}=1$. Setting $\lambda;=hk$, it has the same properties as before.

Now, by the second equation of System \eqref{casee1}, we obtain

\begin{eqnarray*}
    m/\mu&=&-h^{1-q^{s(2t-1)}}\alpha^{q^{s(2t-1)}} k^{1-q^{s}}\alpha^{-\qs}\\&=&-h^{1-q^{s(2t-1)}}\alpha^{q^{s(t-1)}} h^{q^{s(2t-1)}-1} k^{q^{s(t-1)}-1} k^{1-q^{s}}\alpha^{-\qs}\\&=&-\alpha^{q^{s(t-1)}-\qs}  k^{q^{s(t-1)}-q^{s}}\\&=&-\big(\alpha^\qs k^\qs\big)^{q^{s(t-2)}-1}.
\end{eqnarray*}
Finally, note that $(\alpha^{q^s}k^\qs)^{q^{st}}=\alpha^\qs k^{\qs-1}h^{q^{s(t+1)}-1}k^{q^{s(t+1)}}=\lambda^{q^{3s}+1}\alpha^\qs k^\qs$, 
so $m/\mu=-z^{q^{s(t-2)}-1}$ with $z^{q^{st}}=\lambda^{q^{3s}-1}z$.
This concludes the proof.
\end{proof}

\end{document}